\documentclass[11pt, reqno]{amsart} 
\usepackage{amsmath,amsthm,amsfonts,amssymb,mathrsfs,bm,graphicx,stmaryrd,hyperref}
\usepackage[usenames]{color}
\usepackage{graphicx}
\usepackage[letterpaper,hmargin=1.0in,vmargin=1.0in]{geometry}
\parskip 	\smallskipamount

\newtheorem{theorem}{Theorem}[section]
\newtheorem{lemma}[theorem]{Lemma}
\newtheorem{corollary}[theorem]{Corollary}

\newtheorem{remark}[theorem]{Remark}

\newtheorem{definition}[theorem]{Definition}

\DeclareMathOperator*{\argmax}{arg\,max}

\def\N{\mathbb{N}}

\def\Z{\mathbb{Z}}
\def\R{\mathbb{R}}
\def\S{\mathcal{S}}
\def\C{\mathcal{C}}

\def\A{\mathcal{A}}

\def\cL{\mathcal{L}}

\def\E{\mathbb{E}}
\def \e{\varepsilon}

\def \epsilon{\e}

\def \epsilon{\e}

\def\l{\ell}

\newcommand{\PP}{\mathbb{P}}

\newcommand{\eqd}{\stackrel{d}{\implies}}

\begin{document}
\title{The starting point of a directed geodesic is special}

\author{Pantelis Tassopoulos}
\address{Department of Pure Mathematics and Mathematical Statistics, 
University of Cambridge, Cambridge, United Kingdom}
\email{pkt28@cam.ac.uk}

\author{Sourav Sarkar}
\address{Department of Pure Mathematics and Mathematical Statistics, 
University of Cambridge, Cambridge, United Kingdom}
\email{ss2871@cam.ac.uk}

\subjclass[2010]{$82B23$, $82C22$ and $60H15$}
\date{}
\begin{abstract}For any $t\in [0,1/2)$, consider the process $\eta_t:[0,1/2]\mapsto\R$ defined by
$\eta_t(s)=\Pi(t+s)-\Pi(t)$, where $\Pi$ is the directed geodesic from $(0,0)$ to 
$(0,1)$ in the directed landscape.
Let $0\leq t<u<1/2$. We show that the laws of $\eta_t$ and $\eta_u$ are mutually absolutely continuous if and only if $t>0$. This proves Conjecture 14.5 in \cite{DOV} in the affirmative.
\end{abstract}

\maketitle

\maketitle
\tableofcontents
       


\section{Introduction} 
In 1986, Kardar, Parisi and Zhang \cite{kardar1986dynamic} predicted universal scaling behavior for many planar random growth processes. A central object to describe the random growth models in the KPZ universality class is the Airy line ensemble. The parabolic Airy line ensemble is a random sequence of functions $\A_1> \A_2 > \ldots$. It was introduced by Pr\"ahofer and Spohn \cite{prahofer2002scale} in the version $\A_i(t) + t^2$, which is stationary in time, see also Corwin and Hammond \cite{corwin2014brownian}. The functions $\A_i(t)$
are locally Brownian with
diffusion parameter $2$. The top line $\A_1$ is known as the parabolic Airy$_2$ process that appear as the limiting spatial fluctuation of random growth models starting from a single point. The Airy sheet and directed landscape were constructed by Dauvergne, Ortmann and Vir\'ag \cite{DOV} as the scaling limits of one of the KPZ-class models, Brownian last passage percolation. The standard Airy sheet $\S:\R^2\mapsto \R$ is a random continuous function defined in terms of the Airy line ensemble such that $\S(0,\cdot)=\A_1(\cdot)$.
The construction of the directed landscape from the Airy sheet is analogous to that of the Brownian motion from Gaussian distribution. Let $\R^4_\uparrow=\{(x,s;y,t)\in \R^4: s<t\}$. The coordinates $x$ and $y$ can be thought of as spatial and the coordinates $s$ and $t$ as temporal. Then the directed landscape $\mathcal L: \R^4_\uparrow\mapsto\R$ is a random continuous function satisfying
the metric composition law
\begin{align}\label{eq: metric comp 1}
    \cL(x,r;y,t)=\sup_{z\in \R}(\cL(x,r;z,s)+\cL(z,s;y,t))\,,
\end{align}
for all $(x,r,y,t)\in \R^4_\uparrow$ and all $s\in (r,t)$; and with the property that $\cL(\cdot, t;\cdot, t + s^3)$ are independent Airy sheets of scale $s$ for any set of
disjoint time intervals $(t,t+s^3)$.

On the other hand, all models in the KPZ universality class have an analogue of the
height function which is conjectured to converge at large time and small length scales under the KPZ $1:2:3$ scaling to a universal object $h_t(\cdot)$ called the KPZ fixed point. Matetski-Quastel-Remenik \cite{matetski2021kpz} construct the KPZ fixed point as a Markov process in $t$, and they show that it is a limit of the
height function evolution of the totally asymmetric simple exclusion process with arbitrary initial condition. The KPZ fixed point
can also be written in terms of the directed landscape and its initial condition $h_0:\mathbb R \to \mathbb R\cup \{-\infty\}$ as
\begin{align}\label{eq: KPZ fixed pt}
    h_t(y) = \sup_{x\in \R} (h_0(x) + \cL(x, 0; y, t))\,,\qquad y 
    \in \R\,.
\end{align}
The KPZ fixed point started from some arbitrary initial height profile $h_0$ at a positive time $t$ is a random continuous function and exhibits rich fractal structure. In fact, $h_t-h_t(0)$ is mutually absolutely continuous with respect to Brownian motion with diffusivity $2$ on bounded subsets of the line, \cite{sarkar2021brownian} and \cite{tassopoulos2026kpzfixedpointbrownian}.

The metric composition law \eqref{eq: metric comp 1} suggests the directed landscape can be viewed as a random metric on $(1+1)$-dimensional space-time (satisfying a reverse-triangle inequality). Points for which the maximum is attained lie on so called \emph{directed geodesics}. Following the construction in \cite{DOV}, we define directed geodesics in the directed landscape as follows. First, for a continuous path $\pi:[t,s]\mapsto \R$, we define its \emph{length} by
\begin{equation}
    \label{eq: length dir land}\l(\pi)=\inf_{k\in \N}\inf_{t=t_0<t_1<\ldots<t_k=s}\sum_{i=1}^k\cL(\pi(t_{i-1}),t_{i-1};\pi(t_i),t_i)\,.
\end{equation}
We say that $\pi$ is a directed geodesic from $(\pi(t),t)$ to $(\pi(s),s)$ if
\[\l(\pi)=\cL(\pi(t),t;\pi(s),s)\,.\]
For any $u=(x,t;y,s)\in \R^4_\uparrow$, almost surely there exists a unique directed geodesic $\Pi_u$ from $(x,t)$ to $(y,s)$ (see \cite{DOV}, Theorem $12.1$). 

\cite{DOV}~ asked if the starting point of a directed geodesic is special (see \cite{DOV}, Conjecture $14.5$). In this paper, we show that the conjecture is indeed true, that is, we prove the following.

\begin{theorem}\label{thm: main} For any $t\in [0,1/2)$, consider the process $\eta_t:[0,1/2]\mapsto\R$ defined by
\begin{equation}\label{eq: defeta}
\eta_t(s)=\Pi(t+s)-\Pi(t)\,,
\end{equation}
where $\Pi$ is the directed geodesic from $(0,0)$ to 
$(0,1)$ in the directed landscape.
Let $0\leq t<u<1/2$. Then the laws of $\eta_t$ and $\eta_u$ are mutually absolutely continuous if and only if $t>0$.
\end{theorem}

\subsection{Notation}\label{sec: notation}
We now record some notation we will use throughout.

Let $\C[a,b]$ denote the space of all continuous functions $f:[a,b]\mapsto \R$ and $\C_0[a,b]$ denote the space of all continuous functions $f:[a,b]\mapsto \R$ with $f(a)=0$, with the topologies of uniform convergence. We shall denote the space of all continuous functions in $\R$ as $\C$.

The symbols $\cdot \land \cdot , \cdot \lor \cdot$ denote $\min\{\cdot, \cdot \}$ and $\max\{\cdot, \cdot \}$ respectively. 

We say that a random process taking values in $\R^d$, $d\ge 1$ has \textit{rate $v$} (or \emph{diffusivity} $v$) if the quadratic variation of each of its components in any interval $[s,t]$ is equal to $v(t-s)$. From now on, all Brownian motions are rate two unless stated otherwise. 

When in some estimates a constant appears that will depend on some parameters $a,b,c,\cdots$, it will be denoted by $C_{a,b,c,\cdots}$, unless otherwise specified. Constants without subscripts are deemed to be universal. Constants appearing in various estimates may be enlarged or reduced as needed. 

\subsection{Organization of the paper} Section \ref{sec: prelim} recalls some background materials and relevant results that will be used in the paper. In Section \ref{sec: cont} we show that $\eta_t$ and $\eta_u$ are mutually absolutely continuous if $0<t<u<1/2$. The symmetries of the directed landscape, Lemma \ref{lemma: invariance dir land} and its locally Brownian nature in space, Theorem \ref{thm: mut abs cont airy bm} will play a crucial role, concluding the proof of Theorem \ref{thm: cont}. In Section \ref{sec: sing} we prove that $\eta_0$ and $\eta_u$ are singular for any $0<u<1/2$, thereby proving Theorem \ref{thm: main}. The proof of singularity relies on the fact that the asymptotic fluctuations of rescaled geodesic increments near the starting point are different than those of increments near an interior point, Theorem \ref{thm: distinct-displacement-laws}, established using sharp upper tail asymptotics and the large deviation theory for the directed landscape from \cite{DDV24}. The singularity is then proved using mixing properties of the directed landscape, resulting in the intermediate Theorems \ref{thm: cesaro brown-airy} and \ref{thm: cesaro brown-bessel}.

\section{Preliminaries}\label{sec: prelim}

The standard Airy sheet $\mathcal{S}:\R^2\to \R$ is a random continuous function and can be thought of as assigning `distances' between pairs of space-times points $(x, 0)$ and $(y, 1)$. It can be thought of as the restriction of the metric imposed by the directed landscape computing `lengths' between points $(x, 0)$ and $(y, 1)$ in $(1+1)-$dimensional space-time. Its law is defined in terms of the Airy line ensemble $\mathcal{A}$, see Definition 8.1 in \cite{DOV}. Importantly, for us, $\mathcal{S}(0,\cdot)\overset d =\mathcal{A}_1(\cdot)$, where $\A_1$ is the parabolic Airy$_2$ process. The latter enjoys a strong comparison against Brownian motion on compacts. 

\begin{theorem}[Theorem 1.1. in \cite{dauvergne2024wienerdensitiesairyline}]\label{thm: mut abs cont airy bm}
Fix real $a< b$. Then law of the increments of the parabolic Airy$_2$ process,
\[
\mathcal{A}_1(\cdot+a) - \mathcal{A}_1(a)
\]
on paths $\mathcal{C}_0[0, b-a]$ is mutually absolutely continuous with respect to the law of a rate 2 Brownian motion on $[0, b-a]$. Moreover, the Radon-Nikodym derivative against Brownian motion is almost surely bounded (with bounds depending on $a, b$). 
\end{theorem}

The Airy sheet has a global parabolic curvature evidenced by the almost sure pointwise bounds 
\begin{equation}\label{eq: airy shape bnds}
|\mathcal{S}(x,y)+(x-y)^2|\le \mathfrak{C}+c\log^{2/3}(2+|x|+|y|)\,,\quad \text{ for all } x,y\in \R 
\end{equation}
for some universal constant $c>0$ and some $\mathfrak{C}$ satisfying $\mathbb{E}[a^{\mathfrak{C}^{3/2}}]<\infty$ for some $a>1$, \cite{sarkarthrehalves2022}. The parabolically adjusted version $\mathcal{S}(x, y) + (x - y)^2$ is stationary with respect to spatial translations.

Now, one can extend the Airy sheet to all time-ordered pairs of space-time points $(x, s), (y, t)$ with $s< t$ in a consistent way, recovering the directed landscape, which can be characterised as follows. 

\begin{definition}
	\label{def: dir land}
	The directed landscape $\cL:\R^4_\uparrow \to \R$ is the unique random continuous function satisfying
	\begin{enumerate}
		\item (Airy sheet marginals) For any $t\in \R$ and $s>0$ we have
		$$
		\mathcal{\cL}(x, t; y,t+s^3) \overset d= s\mathcal{S}(x/s^2,y/s^2)
		$$
		as random continuous functions of $x, y$ on $\R^2$.
		\item (Independent increments) For any disjoint time intervals $[s_i, t_i]$, $ 1\le i \le k$, the random functions
		$
		\cL(\cdot, s_i ; \cdot, t_i)\,, 1\le i \le k
		$
		are mutually independent.
		\item (Metric composition law) Almost surely, for any $r<s<t$ and $x, y \in \R$ we have that
		\begin{equation}
			\label{eq: metric comp}
			\cL(x,r;y,t)=\max_{z \in \mathbb R} \cL(x,r;z,s)+\cL(z,s;y,t)\,.
		\end{equation}
	\end{enumerate}
\end{definition}

We record the following useful set of symmetries enjoyed by the directed landscape.
\begin{lemma}[Lemma 10.2, \cite{DOV}, Proposition 1.23, \cite{dauvergne2022scalinglimitlongestincreasing}]
	\label{lemma: invariance dir land} We have the following equalities in distribution as random continuous functions from $\R^4_\uparrow$ to $\R$. Here $r, c \in \R$, and $\e > 0$.
	\begin{enumerate}
		\item (Time stationarity)
		$$
		\displaystyle
		\cL(x, t ; y, t + s) \overset d= \cL(x, t + r ; y, t + s + r).
		$$
		\item (Spatial stationarity)
		$$
		\cL(x, t ; y, t + s) \overset d= \cL(x + c, t; y + c, t + s).
		$$
		\item (KPZ rescaling)
		$$
		\cL(x, t ; y, t + s) \overset d=  \e \cL(\e^{-2} x, \e^{-3}t; \e^{-2} y, \e^{-3}(t + s)).
		$$
        \item (Flip symmetry)
        $$
        \cL(x, s; y, t) \overset d = \cL(-x, s; -y, t) \overset d = \cL(y, -t; x, -s)
        $$
	\end{enumerate}
\end{lemma}

Importantly for us, for any fixed choice of endpoint pairs, the maximum in \eqref{eq: metric comp} is almost surely uniquely attained, giving rise to the existence of geodesics with respect to the length functional in \eqref{eq: length dir land}. Existence and uniqueness of $\cL$-geodesics are shown in \cite[Theorem 12.1 and Lemma 13.2]{DOV}. Moreover, the existence and uniqueness of geodesics subject to various boundary conditions (adding in \eqref{eq: metric comp} final and terminal weights $g, h: \R\to \R$ with asymptotic growth conditions) was also established in \cite[Lemma 6.1]{sarkarthrehalves2022}. In particular, Lemma 6.1 in \cite{sarkarthrehalves2022} gives the following. 

\begin{lemma}[Lemma 6.1, \cite{sarkarthrehalves2022}]\label{lemma: geod bdry cond}
    Let $h, g: \R\to \R\cup\{-\infty\}$ be measurable functions that are locally bounded from above. Fix $s < t$ and consider the random field
    \[
    H(x, y) := h(x) + \cL(x, s; y, t) + g(y)\,,\qquad x, y \in \R\,.
    \]
    Suppose that $H$ achieves its maximum almost surely, and that the set of points in $\R^2$ where it is maximized is almost surely bounded. Then almost surely, there exists a unique geodesic $\gamma$ from $(h, s)$ to $(g, t)$.
\end{lemma}

\begin{remark}
    Note by the above result, the existence and uniqueness of geodesics in $\cL$ from $(h, s)$ to $(g, t)$ will be guaranteed for boundary data $h, g:\R \to \R\cup\{-\infty\}$ continuous and real-valued in some $[-b, b]$ and $-\infty$ outside it; we will denote them by $h^b, g^b$ respectively. The symmetries of $\cL$, Lemma, \ref{lemma: invariance dir land} and the shape bounds \eqref{eq: airy shape bnds} mean the maximum will typically be achieved in some compact set. 
\end{remark}

Complementary to the existence of directed geodesics, we clarify the sense in which we call the directed landscape a `(directed) metric'. By the metric composition law \eqref{eq: metric comp}, for all $(x, s; y, t)$ the directed landscape satisfies the reverse triangle inequality
\[
\cL(p; q) + \cL(q; r) \le \cL(p; r)
\]
for all points $p = (x, s), q = (y, t), r = (z, u)$ with $x, y, z\in \R$ and $s < t< u$. It can thus be construed as an element of the space of `directed metrics' in $\R^4_\uparrow$ in the sense that the directed landscape only measures distances between points in space-time with temporal ordering; the reverse triangle inequality can easily be reversed by a sign change, the fact it is this way is only a matter of convention. Following the notation of \cite{DDV24}, we let $\mathcal{E}$ denote the space of all continuous functions $e:  \R^4_\uparrow\to 
\R$ satisfying the reverse triangle inequality
\[
e(p; q) + e(q; r) \le e(p; r)
\]
for all points $p = (x, s), q = (y, t), r = (z, u)$ with $x, y, z\in \R$ and $s < t< u$. We will henceforth equip $\mathcal E$ with the topology of uniform convergence on bounded sets, which is strictly finer than the topology of uniform convergence on compact sets, needing uniform convergence near the boundary $s=t$ in $\R^4_\uparrow$.

As was the case for the directed landscape in \eqref{eq: length dir land}, we can define a notion of length associated to a continuous paths $\gamma : [s, t]\to \R$. For a directed metric $e \in \mathcal E$, define the \emph{length} $\ell_e(\gamma)$ of $\gamma$ with respect to $e$ by
\begin{equation}
	\label{eq: dir metric length}
	\ell_e(\gamma) = \inf_{k \in \N} \inf_{s = r_0 < r_1 < \dots < r_k=t} \sum_{i=1}^k e(\gamma(r_{i-1}), r_{i-1}; \gamma(r_i), r_i)\,.
\end{equation}
An easy upper bound for $\ell_e(\gamma)$ is given by the triangle inequality, which guarantees that $\ell_e(\gamma) \le e(\gamma(s), s; \gamma(t), t)$; we call a path $\gamma$ an \emph{geodesic} with respect to $e$ if this bound is an equality.

By the global shape bounds in Corollary 10.7 in \cite{DOV}, $\cL^{(\e)}$ converges uniformly on bounded sets to the (deterministic) \emph{Dirichlet} (directed) metric
\[
d(x, s; y, t) := -\frac{(y-x)^2}{t-s}\,,\qquad (x, s; y,t) \in \R^4_\uparrow\,,
\]
as $\e \to 0$. Moreover, we obtain for any absolutely continuous $\gamma: [s, t] \to \R$ the length in \eqref{eq: dir metric length} is the negative Dirichlet energy, $\ell_d(\gamma) = -\int^t_s|\gamma'|^2$ which is maximised by the linear path from $(x, s)$ to $(y, t)$.

We will need the following upper tail large deviation result for the diffusively rescaled directed landscape, $\cL^{(\e)}(x, s; y, t) := \e^{2}\mathcal{L}(\e^{-1} x, s; \e^{-1}y, t)$ for $\e > 0$.

\begin{theorem}[Theorem 1.1, \cite{DDV24}]\label{thm: large dev} There exists a lower semi-continuous function $I : \mathcal E \to [0,\infty]$ such that for every Borel measurable  $A\subset \mathcal E$, as $\e\to 0$ we have
	\begin{equation*}   
		\exp((o(1)-\inf_{A^\circ} I)\e^{-3}) \le P( \cL^{(_\e)}\in A) \le \exp((o(1)-\inf_{\bar{A}}I)\e^{-3})\,,
	\end{equation*}
    where $A^\circ, \bar A$ denote the interior and closure of $A$ respectively.
\end{theorem}

In other words (ignoring some properties of the rate function, see \cite{DDV24} for the full statement), the family $(\cL^{(\e)})_\e$ satisfies a large deviation principle with speed $\e^{-3}$ and good rate function $I$. Note our scaling differs from that in \cite{DDV24} as we take $\e^2$ instead of the $\e$ appearing in their statement.

As mentioned above, for the Dirichlet metric the length of an absolutely continuous path is its negative Dirichlet energy. More generally, the metrics in Theorem \ref{thm: large dev} with finite rate can be considered as perturbations of the Dirichlet metric, with certain regions in $\R^2$ gaining additional weight. We will be interested in the fact that the rate function can be easily computed for a certain class of metrics we shall define below.

Given a singular measure $\mu$ on $\R^2$, define a directed metric $e_\mu$ as follows. For a path $\gamma:[s, t] \to \R$, let $\mathfrak{g} \gamma=\{(\gamma(r),r):r\in[s,t]\}$ denote the \textbf{graph} of $\gamma$, and define
\begin{equation}
	\label{eq: emu}
	e_\mu(x,s;y,t)=\sup_\gamma \mu(\mathfrak{g} \gamma)+\ell_d(\gamma)\,, 
\end{equation}
where the supremum is over all $\gamma\in H^1$, the set of all paths $\gamma$ with $\int^t_s |\gamma'|^2 < \infty$ with domain $[s,t]$ satisfying $\gamma(s)=x$, $\gamma(t)=y$. Informally, the metric $e_\mu$ rewards paths that maximise the $\mu$-value of their trace on $\mathrm{supp}(\mu)$, while being penalised for having excessive fluctuations (due to the Dirichlet energy term). The interested reader can consult the discussion in \cite[Section 1.1]{DDV24} for a detailed literature review of the above type of directed metric.

In our application of the large deviation framework established in \cite{DDV24} in Theorem \ref{thm: distinct-displacement-laws}, we will only consider $\mu$ supported on the graph of a piecewise linear function $f: [s, t] \to \R$, see Figure \ref{fig: planted metric} for an illustration. More precisely, we will consider measures of the form
\[
\mu(A) := c\int^t_s \mathbf{1}_{\{(f(u), u)\in A\}}\mathrm{d}u\,,
\]
for some constant $c >0$ and all $A\subseteq\R^2$ Borel measurable. This is an example of what is called a \emph{planted network measure} in \cite{DDV24}. Note the time-marginal of the restricted measure $\mu|_{\mathfrak{g} f}$ has Lebesgue density $\rho_\mu = c$ on $[s, t]$.

\begin{figure}[h]
    \centering
    \includegraphics[width=0.5\linewidth]{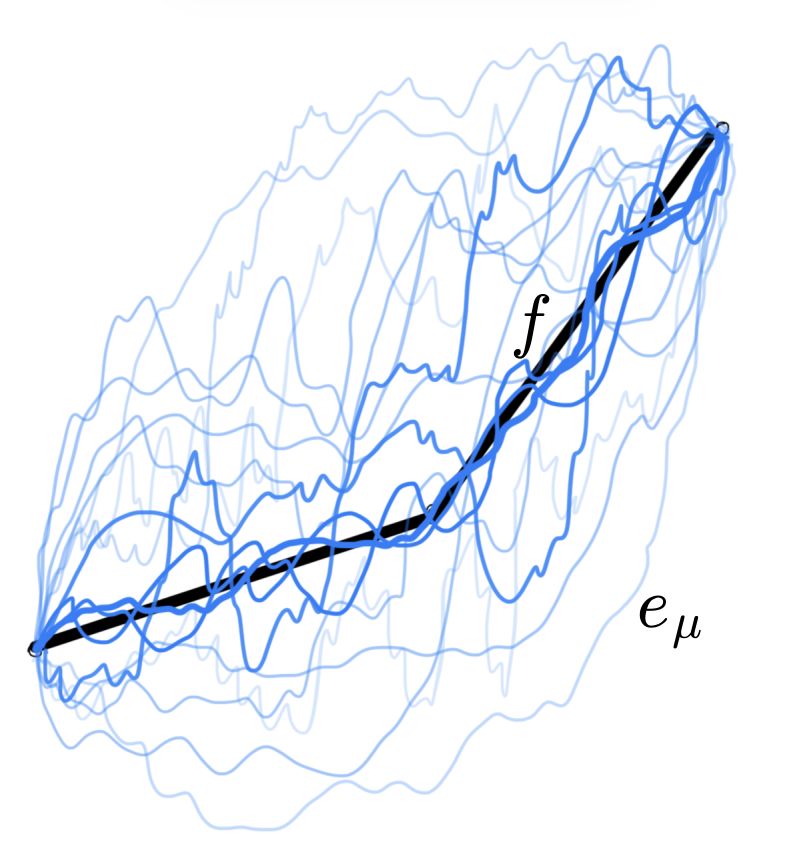}
    \caption{Illustration of a planted metric $e_\mu$ with various finite Dirichlet energy paths (in \color{blue} blue{\color{black}), with intensity proportional to their lengths in $e_\mu$, $\ell_{\mu}(\cdot)$. The length $\ell_{\mu}(\cdot)$ penalises deviation from $f$ due to the measure $\mu$ `planting' $f$ in space-time as well as rapid oscillations, due to the Dirichlet energy present in the construction of $e_\mu$.}}
    \label{fig: planted metric}
\end{figure}
Now, the rate function in \cite[Theorem 1.1]{DDV24} can be computed as
\begin{equation}\label{eq: rate function}
    I(e_\mu)=\frac43 \int_{\R^2} \sqrt{\rho_\mu}d\mu = \frac43 c^{3/2}(t-s)\,.
\end{equation}

\section{Proof of absolute continuity}\label{sec: cont}
In this section, we shall prove the following theorem.
\begin{theorem}\label{thm: cont}
For any $t\in [0,1/2)$, consider the process $\eta_t:[0,1/2]\mapsto\R$ defined in \eqref{eq: defeta}.
Let $0< t<u<1/2$. Then the laws of $\eta_t$ and $\eta_u$ are mutually absolutely continuous.
\end{theorem}
\begin{proof} Let $0< t<1/2$. Consider the random variables 
\[f_t(x):=\cL(0,0;x,t), \qquad g_t(y):=\cL(y,t+1/2;0,1),\qquad h_t(x,s;y,v)=\cL(x,t+s; y,t+v)\,,\]
for $x,y\in \R$ and $0\leq s<v\leq 1/2$. Then $f_t, g_t$ and $h_t$ are independent because of the independent increment property of the directed landscape, see Definition \ref{def: dir land}. Define
\[(X_t,Y_t)=\argmax_{x,y}(f_t(x)+h_t(x,0;y,1/2)+g_t(y))\,,\]
and $\Gamma_t$ be the almost surely unique geodesic from $(X_t,0)$ to $(Y_t,1/2)$ in the directed landscape $h_t$ (guaranteed by Lemma \ref{lemma: geod bdry cond}). Then, we have by inspection the identification
\[\eta_t(\cdot)=\Gamma_t(\cdot)-\Gamma_t(0)\,.\]
For $0<t<u<1/2$, by the time stationarity of the directed landscape, Lemma \ref{lemma: invariance dir land}, 
\[h_t(x,s;y,v)\overset{d}{=}h_u(x,s;y,v)\,,\qquad x,y\in \R\,, 0\leq s<v\leq 1/2\,,\]
as processes in $\R^4_\uparrow$. Also, by Theorem \ref{thm: mut abs cont airy bm} the recentred Airy$_2$ process is mutually absolutely continuous with respect to Brownian motion on any compact set and the Wiener measure is invariant under Brownian scaling, we have that on any compact set, the laws of $f_t-f_t(0)$ and $f_u-f_u(0)$ are mutually absolutely continuous, and so are the laws of $g_t-g_t(0)$ and $g_u-g_u(0)$. 

We now show this local mutual absolute continuity is sufficient to yield that the laws of $\eta_t$ and $\eta_u$ are mutually absolutely continuous on $\mathcal{C}[0, 1/2]$. We will proceed by a tilting argument.

Indeed, for any $b> 0$, on the event $\Gamma_t(0), \Gamma_t(1/2)\in [-b, b]$, we have $\Gamma_t = \Gamma^b_t$, where $\Gamma^b_t$ is the almost surely unique geodesic (by Lemma \ref{lemma: geod bdry cond}) from $(X^b_t,0)$ to $(Y^b_t,1/2)$ in the directed landscape $h_t$, where the endpoints satisfy
\begin{align}\label{eq: maximisers trunc}
    (X^b_t,Y^b_t)=\argmax_{x,y\in \R}(f^b_t(x)+h_t(x,0;y,1/2)+g^b_t(y))\in [-b, b]^2\,,
\end{align}
where for a function $f$, $f^b$ denotes the function agreeing with $f$ on $[-b, b]$ and set to $-\infty$ elsewhere. 

By the above mutual absolute continuity of $f^b_t-f_t(0), f^b_u-f_u(0)$ and $g^b_t-g_t(0), g^b_u-g_u(0)$ on compacts, we have the laws of $\Gamma^b_t, \Gamma^b_u$ are mutually absolutely continuous on $\mathcal{C}[0, 1/2]$. 

Now, for $b> 0$, $k\ge 1$, consider the `tent' function
\[
\varphi_{b, k}(x) = \begin{cases}
    &0 \,,\quad x\le -b-1/k\,,\\
    &k(x+b)+1\,,\quad x\in [-b-1/k, -b]\,,\\
    &1 \,,\quad x\in [-b, b]\,,\\
    &-k(x-b)+1 \,,\quad x\in [b, b+1/k]\,,\\
    &0 \,,\quad x\ge b+1/k\,.
\end{cases}
\]
Observe in particular that $\varphi_{b, k}$ is non-negative everywhere, $\varphi_{b, k} =  1$ on $[-b, b]$ and vanishes outside $[-b-1/k, b+1/k]$. Moreover, each of $X^b_t, Y^b_t$ in \eqref{eq: maximisers trunc} are locations of maximisers of pointwise sums of independent locally Brownian processes,
\begin{align*}
    X^b_t \in \argmax_{x \in [-b, b]}(f_t(x) + j(x))\,,\quad j(x) = \max_{y\in [-b, b]}h_t(x,0;y,1/2)+g_t(y)\\
Y^b_t \in \argmax_{y \in [-b, b]}(k(y)+g_t(y))\,,\quad k(y) = \max_{x\in [-b, b]}f_t(x)+h_t(x,0;y,1/2)
\end{align*}
respectively. Indeed, $j, k$ are independent from $f_t$ and $g_t$ respectively and by the symmetries of the directed landscape, Lemma \ref{lemma: invariance dir land} are equal in law to KPZ fixed points (recall from \eqref{eq: KPZ fixed pt}) started from different initial data, which are locally absolutely continuous with respect to Brownian motion, \cite{sarkar2021brownian}. Thus, $\PP(X^b_t = \pm b) = \PP(Y^b_t = \pm b)=0$ and so the maximisers $(X^b_t, Y^b_t)$ do not lie on the boundary of $[-b, b]^2$ almost surely. 

We now show that the unique geodesics $\Gamma^{n,k, b}_t$ from $(X^{n, k, b}_t, 0)$ to $(Y^{n, k, b}_t, 1/2)$ with
\[
(X^{n, k, b}_t,Y^{n, k, b}_t)=\argmax_{x,y\in \R}(f_t(x)-f_t(0)+n\varphi_{b, k}(x)+h_t(x,0;y,1/2)+g_t(y)-g_t(0)+n\varphi_{b, k}(y))
\]
coincide with the coupled geodesic $\Gamma^{b}_t$ for all sufficiently large $n, k$. Indeed, it suffices to show for all sufficiently large $n, k$,
\begin{align*}
    &\sup_{(x,y)\in \R^2\setminus [-b, b]^2}(f_t(x)+n\varphi_{b, k}(x)+h_t(x,0;y,1/2)+g_t(y)+n\varphi_{b, k}(y))\\
    &< 2n + \max_{x,y\in [-b, b]}(f_t(x)+h_t(x,0;y,1/2)+g_t(y))\,,
\end{align*}
using $\varphi_{b, k} = 1$ on $[-b, b]$. Since the maximisers $(X^b_t, Y^b_t)$ are non-atomic by the preceding discussion, we have by continuity and compactness that almost surely there exists some random positive $\delta< b$ such that $|X^b_t|\lor |Y^b_t|\le b-\delta$ giving (note the \emph{strict} inequality below)
\begin{align*}
    &\sup_{(x,y)\in [-b-\delta, b+\delta]^2\setminus[-b, b]^2}(f_t(x)+h_t(x,0;y,1/2)+g_t(y)) < \max_{x,y\in [-b+\delta, b-\delta]}(f_t(x)+h_t(x,0;y,1/2)+g_t(y))\,.
\end{align*}
Now, taking $k > 1/\delta$, we have using $\varphi_{b, k}\le 1$ on $\R$, that in the annulus $[-b-\delta, b+\delta]^2\setminus[-b, b]^2$
\begin{align*}
    &\sup_{(x,y)\in [-b-\delta, b+\delta]^2\setminus[-b, b]^2}(f_t(x)+n\varphi_{b, k}(x)+h_t(x,0;y,1/2)+g_t(y)+n\varphi_{b, k}(y))\\
    &< 2n+ \max_{x,y\in [-b-\delta, b+\delta]}(f_t(x)+h_t(x,0;y,1/2)+g_t(y))\,.
\end{align*}
Moreover, also observe using that $\varphi_{b, k}$ vanishes outside $[-b-1/k, b+1/k]$ and $\varphi_{b, k}\le 1$ on $\R$, we estimate
\begin{align*}
    &\sup_{(x,y)\in \R^2\setminus[-b-\delta, b+\delta]^2}(f_t(x)+n\varphi_{b, k}(x)+h_t(x,0;y,1/2)+g_t(y)+n\varphi_{b, k}(y))\\
    &\le n+ \sup_{x,y\in \R^2\setminus[-b-\delta, b+\delta]^2}(f_t(x)+h_t(x,0;y,1/2)+g_t(y))\,.
\end{align*}
Finally, in the box $[-b, b]^2$, since $\varphi_{b, k} = 1$ on $[-b, b]$,
\begin{align*}
    &\max_{x,y\in [-b, b]}(f_t(x)+n\varphi_{b, k}(x)+h_t(x,0;y,1/2)+g_t(y)+n\varphi_{b, k}(y))\\
    &= 2n+\max_{x,y\in [-b, b]}(f_t(x)+h_t(x,0;y,1/2)+g_t(y))\\
    &= 2n+\max_{x,y\in [-b+\delta, b-\delta]}(f_t(x)+h_t(x,0;y,1/2)+g_t(y))\,.
\end{align*}
Now, taking $n > \sup_{(x,y)\in \R^2\setminus[-b-\delta, b+\delta]^2}(f_t(x)+h_t(x,0;y,1/2)+g_t(y))-\max_{x,y\in [-b+\delta, b-\delta]}(f_t(x)+h_t(x,0;y,1/2)+g_t(y))$, we obtain the desired conclusion.

Now, suppose $A$ is some event in $\mathcal{C}[0, 1/2]$ for which $\PP(\eta_t\in A) = \PP(\Gamma_t-\Gamma_t(0)\in A) = 0$. By the Brownian Gibbs property for the parabolic Airy$_2$ process, \cite{corwin2014brownian} and the Cameron-Martin theorem for Brownian bridges, for any $n\ge 1$, the laws of the recentred processes $f_t-f_t(0)-n+n\varphi_{b, k}$ and $g_t-g_t(0)-n+n\varphi_{b, k}$ are absolutely continuous with respect to their unperturbed versions $f_t-f_t(0)$ and $g_t-g_t(0)$ respectively, on paths on $\R$. This means the geodesics $\Gamma^{n, k, b}_t$ are absolutely continuous with respect to $\Gamma_t$ for all $n, k , b > 0$. Thus, we have $\PP(\Gamma^{n, k , b}_t-\Gamma^{n, k , b}_t(0)\in A) = 0$ for all $n, k , b> 0$. By the preceding discussion, $\Gamma^{n, k, b}_t$ almost surely eventually (in $k, n)$ coincides with $\Gamma^b_t$. Passing to the limits as $n, k \to \infty$ we thus obtain $\PP(\Gamma^{b}_t - \Gamma^b_t(0) \in A) = 0$ for all $b> 0$. By the aforementioned local mutual absolute continuity of the truncated environments $\eta^b_t = \Gamma^b_t-\Gamma^b_t(0)$ and $\eta^b_u = \Gamma^b_u-\Gamma^b_u(0)$, $\PP(\eta^b_u \in A) = 0$ for $b>0$. Taking $b\to \infty$ gives $\PP(\eta_u\in A) = 0$, concluding the proof of absolute continuity.

The converse is obtained arguing verbatim, replacing $t$ with $u$.
\end{proof}

\section{Proof of singularity}\label{sec: sing} In this section, we will prove the mutual singularity of the laws of the directed geodesic re-centred at its starting point and the geodesic re-centred at any interior point.
\begin{theorem}\label{thm: sing}
For any $t\in [0,1/2)$, consider the process $\eta_t:[0,1/2]\mapsto\R$ defined in \eqref{eq: defeta}. Let $0<u<1/2$. Then the laws on $\mathcal{C}[0, 1/2]$ of $\eta_0$ and $\eta_u$ are singular.
\end{theorem}
\begin{remark}Theorem \ref{thm: cont} and Theorem \ref{thm: sing} together prove Theorem \ref{thm: main}.
\end{remark}

Before proving Theorem \ref{thm: sing}, we need a result establishing that the `environment' near the starting point of a directed geodesic has a scaling limit, which we identify. We state it in slight more generality than needed, considering directed geodesic from $(h_0, 0)$ to $(0, 1) $ where $h_0$ is some measurable function not equal $-\infty$ everywhere with sufficiently rapid decay at infinity. For our purposes, we only need the result for the \emph{narrow wedge} at zero, that is setting $h_0 = 0$ at zero and $-\infty$ elsewhere.

\begin{theorem}[Landscape near the starting point of a geodesic]\label{thm: geo start zero} Let $h_0:\R\to \R\cup\{-\infty\}$ be an upper-semi continuous initial condition with $h_0(0) >-\infty$ satisfying $h_0(x)<c-|x|$ for some $c\in \R$, and let
\[h_\e(x)=\max_{y\in \R}(\e^2 h_0(\e^{-4}y)+\cL(y,0;x,\e^3))\,,\quad \text{ and }\quad  q_\e(x)=\cL(x,2\e^3;0,1)\,.\]
Define
\[f_\e(x):=\e^{-1}(h_\e(\e^2x)-h_\e(0))\,,\qquad g_\e(x):=\e^{-1}(q_\e(\e^2x)-q_\e(0))\,,\]
and
\[\cL_\e(x,s;y,t):=\e^{-1}\cL(\e^2x,\e^3+\e^3s;\e^2y,\e^3+\e^3t)\,, \qquad x,y\in \R, 0\leq s<t\leq 1\,.\]
Then 
\[(f_\e,g_\e,\cL_\e)\eqd (A-\mathcal{A}(0), B,\cL)\,,\qquad \e \to 0\]
where $\mathcal{A}_1$ is the parabolically shifted Airy$_2$ process, $B$ is a rate two Brownian motion, $\cL$ is a directed landscape from time $0$ to $1$, and all three objects are independent.

\end{theorem}
\begin{proof}
 First notice that by translation invariance and $1:2:3$ scaling $\mathcal{L}_{\epsilon}\stackrel{d}{=}\mathcal{L}$ for all $\epsilon\in (0,1/2^{1/3})$. Moreover, by independence of disjoint temporal increments of the directed landscape, it suffices to show $f_{\e}\stackrel{d}{\implies}\mathcal{A}-\mathcal{A}(0)$ and $g_{\e}\stackrel{d}{\implies} B$ separately in distribution in the topology of local uniform convergence as $\epsilon \to 0$.

We start with $(g_\e)_\e$. The $1: 2: 3$ scaling invariance and time stationarity of the directed landscape, Lemma \ref{lemma: invariance dir land}, give for $\e\in (0,1/2^{1/3})$ the distributional equalities on $\mathcal{C}$,
\[
g_{\e}(x) \overset d= \delta^{-1}(\cL(\delta^2 x,0;0,1)-\cL(0,0;0,1))
\]
where $\delta = \e(1-2\e^{3})^{-1/3}$, which tends to zero as $\varepsilon \to 0$. Now, recalling from Section \ref{sec: prelim} that $\cL(\cdot, 0,;0,1)$ is equal in law to the parabolically shifted $\mathrm{Airy}_2$ process, Theorem \ref{thm: mut abs cont airy bm} gives that $\cL(\cdot,0;0,1)-\cL(0,0;0,1)$ has a bounded Radon-Nikodym derivative against rate two Brownian motion on compacts. Now, by \cite[Lemma 4.3]{sarkarthrehalves2022}, we conclude $g_{\e}(\cdot)\stackrel{d}{\implies} B(\cdot)$
with respect to the topology of local uniform convergence on $\R$.

It remains to show $f_{\e}\stackrel{d}{\implies} \mathcal{A}-\mathcal{A}(0)$ as $\epsilon \to 0$. By the $1:2:3$ scaling invariance of the directed landscape, we we have the distributional identity on $\mathcal{C}$
\[
f_{\epsilon}(\cdot)\stackrel{d}{=}\tilde{h}_\e(\cdot)-\tilde{h}_\e(0)
\]
where 
\[
\tilde{h}_\e(x)=\max_{y\in \R}(\e h_0(\e^{-2}y)+\cL(y,0;x,1))\,, \qquad x\in \R\,.
\]

We now show that due to the asymptotic growth condition imposed on the initial condition $h_0$, the rescaled KPZ fixed point $\tilde{f}_\e$ converges almost surely locally uniformly to the re-centred parabolic Airy$_2$ process, equal in law to $\mathcal{L}(0, 0; \cdot, 1)-\mathcal{L}(0, 0; 0, 1)$. In the argument below, uniform constants $c> 0$ will be enlarged or reduced as needed.

Fix any compact set $K\subset \R$, $\delta > 0$. The shape estimates for the stationary Airy sheet \eqref{eq: airy shape bnds} give
\begin{equation*}
\sup_{x\in K}|\cL(x, 0; y, 1)+(x-y)^2|\le \mathfrak{C}+c\log^{2/3}(2+\sup_{x\in K}|x| + |y|)\,,\qquad \text{ for all } y\in \R 
\end{equation*}
for some universal constant $c>0$ and some $\mathfrak{C}$ satisfying $\mathbb{E}[a^{\mathfrak{C}^{3/2}}]<\infty$ for some $a>1$. Moreover, we have for all $\e>0$,
\[
\e h_0(\e^{-2}y) \le c\e - \e^{-1}|y|\,,\qquad y\in \R\,.
\]
Combining the above, we estimate for any $\e \in (0, 1/2^{1/3})$,
\[
\tilde{h}_\e(x)\le \max_{|y|\le \delta}(\e h_0(\e^{-2}y)+\cL(y,0;x,1))\lor \max_{|y|\ge \delta}(\e h_0(\e^{-2}y)+\cL(y,0;x,1))\,,\qquad x \in K\,.
\]
The first term is easily estimated as
\[
\max_{|y|\le \delta}(\e h_0(\e^{-2}y)+\cL(y,0;x,1))\le \max_{|y|\le \delta}\cL(y,0;x,1) + \e c\,,\qquad x\in K\,.
\]
For the second term we have for all $\varepsilon < c^{-1}\land 1/2^{1/3}$
\begin{align*}
   \max_{|y|\ge \delta}(\e h_0(\e^{-2}y)+\cL(y,0;x,1))&\le \max_{|y|\ge \delta}(\mathfrak{C}+c\log^{2/3}(2+\sup_{x\in K}|x| + |y|)+\e c - \e^{-1}|y|)\\
    &\le \max_{|y|\ge \delta}(\mathfrak{C}+c \sup_{x\in K}|x| + (c- \e^{-1})|y|)\\
    &\le \mathfrak{C}+c \sup_{x\in K}|x| + (c- \e^{-1})|\delta|\,,
\end{align*}
which almost surely converges to $-\infty$ as $\e  \to 0$. Thus we have almost surely for all $\e$ sufficiently small (possibly depending on $\delta >0$),
\[
\e h_0(0) + \mathcal{L}(0, 0; x, 1) \le \tilde{h}_\e(x)\le \max_{|y|\le \delta}(\e h_0(\e^{-2}y)+\cL(y,0;x,1))\le \max_{|y|\le \delta}\cL(y,0;x,1) + \e c\,,\qquad x \in K\,.
\]
Taking $\varepsilon\to 0$ we thus obtain,
\[
\limsup_{\e\searrow 0}\sup_{x\in K}|\tilde{h}_\e(x)-\mathcal{L}(0, 0; x, 1)|\le \sup_{x\in K}|\max_{|y|\le \delta}\cL(y,0;x,1)-\mathcal{L}(0, 0; x, 1)|\,.
\]
Since $K, \delta$ were arbitrary, we conclude the proof of local uniform convergence, whence we obtain the desired convergence in distribution on paths in the local uniform topology.
\end{proof}

We need the following technical lemma giving the convergence of the rescaled geodesic endpoints near the starting point.
\begin{lemma}\label{lemma: geod endpts tight} Let $h_0$ be as in the setup of Theorem \ref{thm: geo start zero}. Let $\gamma^\e$ be the geodesic from $(\e^2 h_0(\e^{-4}\cdot), 0)$ to $(0,1)$. Then the families of random variables $\e^{-2}\gamma^\e(\e^3)$ and $\e^{-2}\gamma^\e(2\e^3)$ are tight all $\e\in (0,1/100)$.
\end{lemma}
\begin{proof}
    Observe that for all $\epsilon \in (0,1/2^{1/3})$, the shape bounds of the initial condition $h_0$ implies that $\gamma^\e$ exists and is unique, Lemma \ref{lemma: geod bdry cond}. Moreover, by metric composition and independent increments of the directed landscape, we have the distributional equality
\[
(\e^{-2}\gamma^\e(\e^3), \e^{-2}\gamma^\e(2\e^3)) \stackrel{d}{=} (X_\e, Y_\e):= \argmax_{x,y\in \R}(f_{\epsilon}(x)+g_{\epsilon}(y)+\cL_{\epsilon}(x,0;y,1))
\]
with $f_\e, g_\e, \cL_\e$ as in the statement of Theorem \ref{thm: geo start zero}.

To establish the tightness of $(X_\e, Y_\e)$, by Prokhorov's theorem it suffices to show that for any $\delta> 0$, there exists some $K>0$ such that
\begin{align}\label{eq: tight env}
    \sup_{0< \e < 1/100}\PP(|X_\e|\lor |Y_\e|\ge K)\le \delta\,.
\end{align}

Recall from Theorem \ref{thm: geo start zero} the distributional equality,
\[
g_{\e}(x) \overset d= \delta^{-1}(\cL(\delta^2 x,0;0,1)-\cL(0,0;0,1))\,,\qquad x\in \R
\]
where $\delta = \e(1-2\e^{3})^{-1/3}$, which tends to zero as $\varepsilon \to 0$. By \cite[Theorem 1.1]{dauvergne2024wienerdensitiesairyline}, we have for any $K> 0$ that on $[-\delta^{-2}K, \delta^{-2}K]$, the process $g_\e(\cdot)$ has a Radon-Nikodym derivative with respect to the rate two Wiener measure on that interval that is almost surely uniformly bounded in $\delta$ (though possibly depending on $K$). This means that on the event $Y_\e\in [-K\delta^{-2}, K\delta^{-2}]$, we have the absolute continuity (denoted by the relation $\cdot \ll \cdot$) of environments
\[
(f_\e, g_\e, \cL_\e(\cdot, 0; \cdot, 1)) \ll (f_\e, W, \cL_\e(\cdot, 0; \cdot, 1))
\]
where $W$ is a rate two Brownian motion starting from the origin independent from $(f_\e, \cL_\e)$, where the Radon-Nikodym derivative is uniformly bounded in $\e$ (again, possibly depending on $K$). We now reduce the proof of tightness, \eqref{eq: tight env} to the proof of tightness of the almost surely unique (by Lemma \ref{lemma: geod bdry cond}) maximisers $(X'_\e, Y'_\e)$ of the environments $(f_\e, W, \cL_\e(\cdot, 0; \cdot, 1))$ as $\e \to 0$.

Indeed, fix $\eta> 0$. By the shape bounds of $h_0$ and those of the directed landscape \eqref{eq: airy shape bnds}, $\gamma^\e(0)$ is tight in $\e<1/100$. For any $M>0$, on the event $|\gamma^\varepsilon(0)|<M$, by geodesic monotonicity, $\gamma^\varepsilon$ lies between the geodesics from $(\pm M,0)$ to $(0,1)$. Consequently, by the preceding discussion and Proposition 12.3 \cite{DOV} $\sup_{s\in[0,1]}|\gamma^\varepsilon(s)|$ is tight in $0< \e< 1/100$.  We can now choose $K_\eta> 0$ such that for all $0< \varepsilon < 1/100$, $Y_\e\in [-K_\eta\delta^{-2}, K_\eta\delta^{-2}]$ with probability at least $1-\eta/2$. Let $1 \le C_\eta < \infty$ denote the essential supremum of the Radon-Nikodym derivative of the $\cL(\cdot, 0; 0, 1) -\cL(0, 0 ; 0, 1)$ against rate two Brownian motion on $[-K_\eta, K_\eta]$ (guaranteed to exist by \cite[Theorem 1.1]{dauvergne2024wienerdensitiesairyline}). By the assumed tightness of $(X'_\e, Y'_\e)$, we have there exists $K'_\eta> K_\eta$ such that 
\[
\sup_{0<\e<1/100}\PP(|X'_\e|\lor |Y'_\e|\ge K'_\eta)\le \eta/2C_\eta\,.
\]
We now estimate
\begin{align*}
    &\limsup_{\substack{0<\e<1/100\\\e\downarrow 0}}\PP(|Y_\e|\le  K_\eta\delta^{-2}\,, |X_\e|\lor |Y_\e|\ge K'_\eta)\\
    &\le\limsup_{\substack{0<\e<1/100\\\e\downarrow 0}}\PP\left(\max_{\substack{|x|\lor|y|\le K'_\eta\,,\\|y|\le K_\eta\delta^{-2}}}f_\e(x)+g_\e(y)+\cL(x, 0; y, 1)\le \max_{\substack{|x|\lor|y|\ge K'_\eta\,,\\|y|\le K_\eta\delta^{-2}}}f_\e(x)+g_\e(y)+\cL(x, 0; y, 1)\right)\\
    &\le C_{\eta} \cdot \limsup_{\substack{0<\e<1/100\\\e\downarrow 0}}\PP\left(\max_{\substack{|x|\lor|y|\le K'_\eta\,,\\|y|\le K_\eta\delta^{-2}}}f_\e(x)+W(y)+\cL(x, 0; y, 1)\le \max_{\substack{|x|\lor|y|\ge K'_\eta\,,\\|y|\le K_\eta\delta^{-2}}}f_\e(x)+W(y)+\cL(x, 0; y, 1)\right)\\
    &\le\limsup_{\substack{0<\e<1/100\\\e\downarrow 0}}C_{\eta}\cdot \PP(|Y'_\e|\le K_\eta\delta^{-2}\,, |X'_\e|\lor |Y'_\e|\ge K'_\eta)+ C_{\eta}\cdot \PP(|Y'_\e|\ge K_\eta\delta^{-2})\le \eta\,,
\end{align*}
since the tightness of $Y'_\e$ gives the $\lim_{0<\e<1/100} \PP(|Y'_\e|\ge K_\eta\delta^{-2}) = 0$. This means there exists some $0<\e_0< 1/100$ such that 
\[
 \sup_{0< \e < \e_0}\PP(|X_\e|\lor |Y_\e|\ge K)\le 2\eta\,.
\]
Now, since for $\e_0< \e < 1/100$, one can prove tightness of the maximisers of the environments $(f_\e, g_\e, \cL_\e)$ directly using the global shape estimates \eqref{eq: airy shape bnds} (essentially proceeding as in the argument below), we conclude after possibly enlarging $K$ that
\[
\sup_{\e_0< \e < 1/100}\PP(|X_\e|\lor |Y_\e|\ge K)\le 2\eta\,,
\]
Combining the above estimates gives \eqref{eq: tight env}.

We now prove the tightness of $(X'_\e, Y'_\e)$ using the global shape estimates for the directed landscape, the initial data $h_0$ and Brownian motion. Absolute constants $a, c, d, C$ as well as random constants $\mathfrak{C}_{i, \e}$ may be enlarged or reduced by absolute factors as needed throughout the remainder of the argument. 

Fix $\delta > 0$ and $0< \varepsilon < 1/2^{1/3}$. Since $X'_\e, Y'_\e$ are maximisers of the environment $(f_\e, W, \cL_\e(\cdot, 0; \cdot, 1))$, we have the shape estimates for the directed landscape (and recalling that~$f_\e(0) = W(0) =~0$,~$\cL_\e \overset d=~\cL$),
\begin{align}\label{eq: max tight 1}
    -\mathfrak{C}_{1, \e} \le \cL_\e(0, 0; 0, 1) &\le f_{\epsilon}(X'_\e)+W(Y'_\e)+\cL_{\epsilon}(X'_\e,0;Y'_\e,1)\nonumber\\
    &\le \mathfrak{C}_{1, \e} + f_{\epsilon}(X'_\e)+W(Y'_\e) + c \log^{2/3}(2+|X'_\e|+|Y'_\e|)-(X'_\e-Y'_\e)^2\,,
\end{align}
where
$\mathfrak{C}_{1, \e}$ is a random constant with $\sup_{0< \e < 1/2^{1/3}}\E a^{\mathfrak{C}^{3/2}_{1, \e}} < \infty$ for some universal constants $a > 1$, $c>0$. 

We now estimate from above using the elementary inequalities for all $x, y \in \R$ and some absolute constant $c> 0$
\[
|y|+(x -y)^2 \ge c(|x|+|y|-1)\,, \quad \log(2+|x|+|y|)\le \log2 + \log(1+|x|)+\log(1+|y|)\,,
\]
\begin{equation}\label{eq: max tight 2}
\begin{array}{ll}
     f_\e(X'_\e) &\displaystyle \le \mathfrak{C}_{2, \e} + \max_{y\in \R}(c-|y|-(y-X'_\e)^2 + c \log^{2/3}(2+|X'_\e|+|y|))\\[1ex]
     &\displaystyle \le \mathfrak{C}_{2, \e} + \max_{y\in \R}(-c(|X'_\e|+|y|) + c\log2 + c\log(1+|X'_\e|)+c\log(1+|y|))\\[1ex]
     &\displaystyle \le \mathfrak{C}_{2, \e}+ c\log(1+|X'_\e|)-c|X'_\e|+ c\max_{y\in \R}(\log(1+|y|)-|y|)\\[1ex]
     &\displaystyle \le \mathfrak{C}_{2, \e} + c\log(1+|X'_\e|)-c|X'_\e|\,,
\end{array}
\end{equation}
for some $\mathfrak{C}_{2, \e}>0$ such that $\sup_{0<\e<1/2^{1/3}}\mathbb{E}[a^{\mathfrak{C}^{3/2}_{2, \e}}]< \infty$ and some universal $a> 1$, $c> 0$. 

Combining \eqref{eq: max tight 1} and \eqref{eq: max tight 2}, we obtain
\begin{align*}
    -\mathfrak{C}_{1, \e} &\le \mathfrak{C}_{1, \e} + \mathfrak{C}_{2, \e} +W(Y'_\e)+ c \log(1+|Y'_\e|)+c\log(1+|X'_\e|)-c|X'_\e|-(X'_\e-Y'_\e)^2\\
    &\le \mathfrak{C}_{1, \e} + \mathfrak{C}_{2, \e} + W(Y'_\e) -c(|X'_\e|+|Y'_\e|)\,.
\end{align*}
We thus have for any $K> 1$ the inclusion
\[
\{K\le |X'_\e|\lor |Y'_\e|\le 2K\}\subseteq \left\{2\mathfrak{C}_{1, \e}+\mathfrak{C}_{2, \e}+\sup_{|v| \le 2K}W(v)\ge cK\right\}\,,
\]
where $c> 0$ is some universal constant. Observe the tails of $\mathfrak{C}_{1, \e}, \mathfrak{C}_{2, \e}, \sup_{|v|\le 2r}W(v)$ satisfy
\[
\sup_{0< \e < 1/100}\PP(\mathfrak{C}_{1, \e}>r) + \PP(\mathfrak{C}_{2, \e}>r) +  \PP\left(\sup_{|v|\le 2r}W(v)\ge ar\right) \le C\mathrm{e}^{-dr}\,, \qquad r> 0\,,
\]
for universal $a, d, C> 0$. Summing over dyadic annuli gives
\[
\sup_{0<\e<1/100}\PP(|X'_\e|\lor |Y'_\e|\ge K)\le C\mathrm{e}^{-dK}\,,
\]
obtaining \eqref{eq: tight env} for $K$ sufficiently large depending on $\delta> 0$, thus concluding the proof of tightness.
\end{proof}

Finally, piecing Theorem \ref{thm: geo start zero} and Lemma \ref{lemma: geod endpts tight}, we identify the distributional limit of rescaled geodesic increments near its starting point.
\begin{theorem}\label{thm: geod conv env} Let $h_0$ be as in the setup of Lemma \ref{lemma: geod endpts tight}. Let $\gamma^{\e}$ be the geodesic from $\e^2 h_0(\e^{-4}\cdot)$ to $(0,1)$. Then,
\[\e^{-2}(\gamma^{\e}(2\e^3)-\gamma^{\e}(\e^3))\eqd Y-X\]
as $\e \to 0$, where 
\[(X,Y):=\argmax_{x,y}(\mathcal{A}_1(x)-\mathcal{A}_1(0)+B(y)+\cL(x,0;y,1))\,,\]
where $\mathcal{A}_1$ is the parabolically shifted Airy$_2$ process, $B$ is a rate two Brownian motion, $\cL$ is a directed landscape from time $0$ to $1$, and all three objects are independent. We will call the tuple $(\mathcal{A}_1-\mathcal{A}_1(0), B, \cL(\cdot, 0; \cdot, 1)$ the \textbf{Brown-Airy environment}. In fact, even more is true, as we have joint convergence, 
\[(\e^{-2}\gamma^\e(\e^3), \e^{-2}\gamma^\e(2\e^3))\eqd (X,Y)\,,\qquad \e\to 0\,.\]
\end{theorem}
\begin{proof}
As in the proof of Lemma \ref{lemma: geod endpts tight}, we have for all $\epsilon \in (0,1/2^{1/3})$, 
\[
(\e^{-2}\gamma^\e(\e^3), \e^{-2}\gamma^\e(2\e^3)) \stackrel{d}{=} (X_\e, Y_\e):= \argmax_{x,y\in \R}(f_{\epsilon}(x)+g_{\epsilon}(y)+\cL_{\epsilon}(x,0;y,1))
\]
with $f_\epsilon, g_{\epsilon}, \cL_{\epsilon}$ as in Theorem \ref{thm: geo start zero}. By Skorokhod's representation theorem and Theorem \ref{thm: geo start zero}, we can construct a coupling on some underlying probability space such that (passing to a subsequence) we have almost surely,
\[
(f_{\e_n}, g_{\e_n}, \cL_{\e_n})\longrightarrow (\mathcal{A}_1 -\mathcal{A}_1(0), B, \cL)\,,\qquad n\to \infty
\]
in the topology of local uniform convergence, where $\mathcal{A}_1$ is the parabolically shifted Airy$_2$ process, $B$ is a rate two Brownian motion, $\cL$ is a directed landscape from time $0$ to $1$, and all three objects are independent. Moreover, by Lemma \ref{lemma: geod endpts tight}, $(X_\e, Y_\e)$ are tight and so another application of Skorokhod's representation theorem gives (again up to a further subsequence),
\begin{equation}\label{eq: skorokhod geod endpts}
    (X_{\e_n}, Y_{\e_n}, f_{\e_n}, g_{\e_n}, \cL_{\e_n})\longrightarrow (X, Y, \mathcal{A}_1 -\mathcal{A}_1(0), B, \cL)\,,\qquad n\to \infty
\end{equation}
in the product topology, almost surely. Now, for all $n$ we have
\[
f_{\e_n}(X_{\e_n})+g_{\e_n}(Y_{\e_n})+\cL_{\e_n}(X_{\e_n},0;Y_{\e_n},1) = \max_{x, y \in \R}f_{\e_n}(x)+g_{\e_n}(y)+\cL_{\e_n}(x,0;y,1)\,.
\]
By \eqref{eq: skorokhod geod endpts}, we have the left hand side converges to
$\mathcal{A}_1(X)-\mathcal{A}_1(0)+B(Y)+\cL(X,0;Y,1)$. Now the almost sure convergence of the unique maximisers $(X_{\e_n}, Y_{\e_n})$ as $n\to \infty$ and the attainment of the maximum (which is also unique, by Lemma \ref{lemma: geod bdry cond}) in the environment $(\mathcal{A}_1 -\mathcal{A}_1(0), B, \cL)$ gives some random $K$ such that almost surely for all $\e_n$,
\[
\max_{x, y \in \R}f_{\e_n}(x)+g_{\e_n}(y)+\cL_{\e_n}(x,0;y,1) = \max_{|x|\lor |y| \le K}f_{\e_n}(x)+g_{\e_n}(y)+\cL_{\e_n}(x,0;y,1)
\]
and
\[
\max_{x, y \in \R}\mathcal{A}_1(x) -\mathcal{A}_1(0)+ B(y) + \cL(x, 0; y, 1) = \max_{|x|\lor |y| \le K}\mathcal{A}_1(x) -\mathcal{A}_1(0)+ B(y) + \cL(x, 0; y, 1)\,.
\]
Now, the local uniform convergence of the environments in \eqref{eq: skorokhod geod endpts} and the above gives 
\[
(X, Y) = \operatorname*{argmax}_{x, y\in \R} \mathcal{A}_1(x)-\mathcal{A}_1(0)+B(y)+\cL(x,0;y,1)\,,
\]
concluding the proof.
\end{proof}

Having established and identified the scaling limit of geodesic increments near its starting point to the Brown-Airy environment in Theorem \ref{thm: geod conv env}, we obtain using mixing properties of the directed landscape a certain `law of large numbers' for the rescaled geodesic increments near its starting point.
\begin{theorem}\label{thm: cesaro brown-airy} Let $\Pi: [0, 1]\to \R$ be the almost surely unique directed geodesic from $(0, 0)$ to $(0, 1)$. Let $\phi:\R\mapsto \R$ be any bounded continuous function. Then, with \begin{equation}\label{eq: at0}
F(n):=n^{-1}\sum_{i=1}^n\phi(2^{2i}(\Pi(2^{-3i+1})-\Pi(2^{-3i})))\,,
\end{equation}
we have almost surely, 
$$F(n)\to \mathbb{E}\phi(Y-X)$$
with
\[
(X, Y) = \operatorname*{argmax}_{x, y\in \R} \mathcal{A}_1(x)-\mathcal{A}_1(0)+B(y)+\cL(x,0;y,1)\,,
\]
where $\mathcal{A}_1$ is the parabolically shifted Airy$_2$ process, $B$ is a rate two Brownian motion, $\cL$ is a directed landscape from time $0$ to $1$, and all three objects are independent.
\end{theorem}
\begin{proof}
By Theorem \ref{thm: geod conv env} taking $h_0$ to be such that $h_0(0) = 0$ and $-\infty$ elsewhere, we have
\[
2^{2i}(\Pi(2^{-3i+1})-\Pi(2^{-3i}) \eqd Y-X\,,\qquad i\to \infty\,.
\]
Equivalently, for any $\phi$ bounded continuous we have the convergence of expectations
\[
\mathbb{E}\phi(2^{2i}(\Pi(2^{-3i+1})-\Pi(2^{-3i}))) \to \mathbb{E}\phi(Y-X)\,,\qquad i \to \infty \,.
\]
Thus, the expectation of the C\'{e}saro means $\mathbb{E}F(n)$ converge to $\mathbb{E}\phi(Y-X)$ as $n\to \infty$. 

Now, let $\Pi_\infty$ be the almost surely unique infinite geodesic in the directed landscape $\cL$ from $(0, 0)$ with direction $0$, \cite[Theorem 3.12, Theorem 3.14]{RahmanVirag2025}. This is a random path which is uniquely characterised by the following (Definition 3.14 in \cite{RahmanVirag2025}):
\begin{enumerate}
    \item $\Pi_\infty:[0, \infty) \to \R$ is continuous and starts at $(0, 0)$,
    \item satisfies $\lim_{t\to \infty} \Pi_\infty(t)/t = 0$
    \item for all $0\le s < t< \infty$, the restriction of $\Pi_\infty$ to $[s, t]$ is an $\cL$-geodesic, meaning
    \[
    \cL(\Pi_\infty(s),s;\Pi_\infty(t),t) = \cL(\Pi_\infty(s),s;\Pi_\infty(r),r)+ \cL(\Pi_\infty(r),r;\Pi_\infty(t),t)\,,\quad s < r < t\,.
    \]
\end{enumerate}
By \cite[Theorem 3.12]{RahmanVirag2025}, it can be constructed as a pointwise limit of the geodesics $\Pi_n$ in the directed landscape $\cL$ from $(0, 0)$ to $(0, n)$ as $n\to \infty$. We can thus represent $\Pi_\infty = H(\cL)$, where $H:\C(\R^4_\uparrow)\to \mathcal{C}[0, \infty)$ is a measurable functional. Thus, by the uniqueness of semi-infinite geodesics in the directed landscape, (see the paragraph preceding the statement of \cite[Theorem 1.3]{RahmanVirag2025}), we conclude that for any fixed $\e>0$ almost surely, $\e^{-2}\Pi_\infty(\e^3\cdot) = H(\cL^{(\e)})$, where $\cL^{(\e)}(x, s; y, t) = \e^{-1}\cL(\e^2x, \e^3 s; \e^2 y, \e^3 t)$ is the rescaled directed landscape. By $1:2:3$ scaling invariance of the directed landscape, we thus conclude that $\Pi_\infty(\cdot) \overset d= \e^{-2}\Pi_\infty(\e^3\cdot)$ on paths for all $\e> 0$.

Moreover, we have by \cite[Corollary 3.24]{RahmanVirag2025} and the symmetries of the directed landscape, Lemma \ref{lemma: invariance dir land} (see also the discussion at the bottom of \cite[p. 7]{RahmanVirag2025}), the distributional equality
\begin{equation}\label{eq: semi-inf law}
    \Pi_\infty|_{[0, 1]}(\cdot) \overset d = \Gamma\,, \quad  \Gamma \text{ is the geodesic from }(0, 0) \text{ to } (W, 1)\,,
\end{equation}
where $W$ is a two-sided rate two Brownian motion with $W(0) = 0$ independent from the directed landscape used to construct $\Gamma$.

Arguing in the same way as in Theorem \ref{thm: cont} and using the distributional equality \eqref{eq: semi-inf law}, we have that the restriction of the law on $\mathcal{C}[0, 1/2]$ of the restricted geodesic $\Pi|_{[0, 1/2]}$ is absolutely continuous with respect $\Pi_\infty|_{[0, 1/2]}$.

Thus, it suffices to show that almost surely, 
\[
F_\infty(n):=n^{-1}\sum_{i=1}^n\phi(2^{2i}(\Pi_\infty(2^{-3i+1})-\Pi_\infty(2^{-3i})))\to c\,,\qquad n\to \infty\,.
\]
for some constant $c$. Indeed, observe for all $i \ge 1$, the rescaled geodesic increments can almost surely be expressed as
\[ 2^{2i}\Pi_\infty(2^{-3i+1}) - 2^{2i}\Pi_\infty(2^{-3i}) = H(T^i(\cL))(2) - H(T^i(\cL))(1)\,,\qquad i \ge 1\,,\]
where
\[
T: \C(\R^4_\uparrow) \to \C(\R^4_\uparrow): f(x, s; y, t) \mapsto  2f(x/4, s/8; y/4, t/8)\,,
\]
and $T^i$ denotes the $i$-fold composition of $T$. The $1:2:3$ scaling invariance of the directed landscape also means the transformation $T$ leaves the law of $\cL$ invariant. Moreover, by the independence of disjoint temporal slices of the directed landscape, a standard measure theoretic argument gives the measure-preserving transformation $T$ is mixing, and so \emph{ergodic}. Thus, Birkhoff's theorem gives for any bounded measurable $\phi$,
\[
F_\infty(n)  = \frac{1}{n}\sum^n_{i=1}\phi(H(T^i(\cL))(2) - H(T^i(\cL))(1)) \to \E\phi\big(H(T^1(\cL))(2) - H(T^1(\cL))(1)\big)\,,\qquad n\to \infty
\]
almost surely, concluding the proof.
\end{proof}

Also, from \cite[Theorem 7.4]{sarkarthrehalves2022}, we have the existence of a scaling limit of geodesic increments at an interior point in terms of the \emph{Brown-Bessel environment}. Again, using mixing properties of the directed landscape we obtain certain averaging functionals of these rescaled interior geodesic increments (having the same form as in Theorem \ref{thm: cesaro brown-airy}) cannot converge along a fixed deterministic subsequence to any constant different from one depending on the Brown-bessel environment, which we will identify.

\begin{theorem}\label{thm: cesaro brown-bessel} Let $\Pi: [0, 1]\to \R$ be the almost surely unique directed geodesic from $(0, 0)$ to $(0, 1)$. Set 
\[(X_1,Y_1):=\argmax_{x,y\in \R}(-R(x)+B(x)+\cL(x,0;y,1)-R(y)-B(y))\,,\]
where $R$ a two-sided Bessel-3 process, $B$ a two-sided Brownian motion (both with diffusivity one), $\cL$ a directed landscape from time $0$ to $1$ and the three objects independent. We call the tuple $(-R+B, -R-B, \cL(\cdot, 0; , \cdot, 1))$ the \textbf{Brown-Bessel environment}. Moreover, let $\phi:\R\to \R$ be any bounded continuous function and fix $0<t<1$.

Fix any bounded continuous $\phi\in \C(\R)$ and $c\neq \E\phi(Y_1-X_1)$. Then, with $i_0(t) \ge 1$ such that $2^{1-3i_0(t)}+t < 1$ and
\begin{equation}\label{eq: cesaro 2}
G(n):= \frac{1}{n}\sum_{i=i_0(t)}^{n}\phi(2^{2i}(\Pi(2^{-3i+1}+t)-\Pi(2^{-3i}+t)))\,,
\end{equation}
for any fixed subsequence $(n_k)_k$ with $n_k\to \infty$ as $k\to \infty$, almost surely $G(n_k)$ does not converge to $c$ as $k\to \infty$.
\end{theorem}

\begin{proof}
    Fix $\phi$ in $\mathcal{C}$ bounded. By the boundedness of $\phi$, we make the harmless identification $i_0(t) = 1$.
    To obtain \eqref{eq: cesaro 2} first establish the asymptotic independence between the local limits of rescaled increments of the geodesic $\Pi$ in $\cL$ from $(0, 0)$ to $(0, 1)$ at the interior point $t\in (0, 1)$, $2^{2i}(\Pi(2^{-3i+1}~+t)-\Pi(2^{-3i}+t))$, $i\ge 1$ and the directed landscape $\cL(\cdot, r; \cdot, s)$, at times $0\le r < s \le 1$ as $i\to \infty$. 

    More precisely, we will prove that for any bounded continuous $\phi$ and bounded random $H$, measurable with respect to the sigma algebra generated by slices of the directed landscape $\cL(\cdot, s; \cdot, r)$ for $0\le s < r\le 1$,
    \begin{align}\label{eq: asymp brown-bessel mixing}
        \lim_{i\to \infty}\mathbb{E}[H\cdot\phi(Z_i)] = \E\phi(Y_1-X_1)\cdot \mathbb{E} H\,,
    \end{align}    
    where $Z_i = 2^{2i}(\Pi(2^{-3i+1}+t)-\Pi(2^{-3i}+t))$, $i\ge 1$.

    By continuity and a standard measure theoretic argument (essentially an application of the martingale convergence theorem to the bounded martingale $\mathbb{E}[H|\mathcal{F}_\delta]$ as $\delta \searrow 0$ with $\mathcal{F}_\delta$ as below), to establish \eqref{eq: asymp brown-bessel mixing}, it suffices to show that for any $\delta > 0$ and $H_\delta$ bounded measurable with respect to the sigma algebra $\mathcal{F}_\delta$ of the directed landscape $\cL$ generated by time slices (recall $t$ is fixed) $0\le s < r\le t-\delta$, $t+\delta \le s < r\le  1$, we have 
    \begin{align}\label{eq: asymp brown-bessel mixing nbhd}
        \lim_{i\to \infty}\mathbb{E}[H_\delta\cdot\phi(Z_i)] = \E\phi(Y_1-X_1)\cdot \mathbb{E} H_\delta\,.
    \end{align} 
    
    Let $m:= \delta/2$. Conditioning on the sigma algebra with $\e_i = 2^{-i}$,
    \[
    \mathcal{F}^{i, m, t}:=\sigma\big( \{\cL(\cdot, s; \cdot , r) : 0\le s < r\le t+\e^3_i-m, \text{ or } t+2\e_i^3+m \le s < r\le  1 \big)\,,
    \]
    since the random variables $H_\delta$ are $\mathcal{F}^{i, m, t}$-measurable for all $i\ge 1$ sufficiently large, it suffices to show
    \[
    \mathbb{E}\big[ \phi(Z_i)\big\vert \mathcal{F}^{i, m, t}\big] \to \E\phi(Y_1-X_1)\,,\qquad i \to \infty
    \]
    in $L^1$ as $i\to \infty$. Observe that the rescaled geodesic points $\Pi(\e^3_i+t)$, $\Pi(2\e^3_i+t)$ can be expressed as the interior points $\Gamma(\e^3_i+t), \Gamma(2\e^3_i+t)$ of the unique $\cL$ geodesic $\Gamma$ on $[t+\e^3_i-m, t+2\e_i^3+m]$ with endpoints $X^{i, m, t} := \Gamma(t+\e^3_i-m), Y^{i, m, t} := \Gamma(t+2\e_i^3+m)$ given by
    \[
    (X^{i, m, t}, Y^{i, m, t}) = \operatorname*{argmax}_{x, y\in \R}\cL(0, 0; x, t+\e^3_i-m)+\cL(x, t+\e^3_i-m; y, t+2\e_i^3+m)+ \cL(y, t+2\e_i^3+m; 0, 1)\,.
    \]
    By the invariance of the directed landscape under translation and scaling, Lemma \ref{lemma: invariance dir land}, translating by $-t-\e^3_i+m$ and rescaling time by $m$, $\Pi(\e^3_i+t)$, $\Pi(2\e^3_i+t)$ have the same law as the rescaled interior points $m^{2/3}\Gamma^{i, m, t}(1), m^{2/3}\Gamma^{i, m, t}(1+m^{-1}\e^3_i)$ of the unique $\cL$ geodesic $\Gamma^{i, m, t}$ on $[0, 2+m^{-1}\e^3_i]$ with endpoints $\widetilde{X}^{i, m, t} := \Gamma^{i, m, t}(0), \widetilde{Y}^{i, m, t} := \Gamma^{i, m, t}(2+m^{-1}\e^3_i)$ given by
    \[
    (\widetilde{X}^{i, m, t}, \widetilde{Y}^{i, m, t}) = \operatorname*{argmax}_{x, y\in \R}\big(\widetilde{H}^{i,m,t}(x)+\cL(x, 0; y, 2+m^{-1}\e^3_i) + \widetilde{G}^{i, m, t}(y)\big)\,,
    \]
    for $\widetilde{H}^{i,m,t}(x) = m^{-1/3}H^{i, m, t}(m^{2/3}x)$,  $\widetilde{G}^{i,m,t}(y) = m^{-1/3}G^{i, m, t}(m^{2/3}y)$ and
    \begin{align*}
        H^{i, m, t}(x) &:= m^{1/3}\cL(0, -m^{-1}(t+\e^3_i-m); m^{-2/3}x, 0)\,,\\
        G^{i, m, t}(y)&:= m^{1/3}\cL(m^{-2/3}y, 2+m^{-1}\e^3_i; 0, m^{-1}(1-t-\e^3_i+m))\,, \quad x, y \in \R\,,
    \end{align*}
    which are independent from the interior landscape $\cL(\cdot, s; \cdot, r)$ for all $0\le s < r \le 2+m^{-1}\e^3_i$. Taking any fixed $0< m < (1-t)/2\land t/4$ gives
    \[
    \liminf_{i\to \infty}m^{-1}(1-t-2\e^3_i+m)-2 >1\,,\text{ and }\liminf_{i\to \infty}m^{-1}(t+\e^3_i-m) > \frac{t}{2m}>0\,.
    \]
    Hence, by Theorem \ref{thm: mut abs cont airy bm}, we have the Radon-Nikodym derivatives of $H^{i, m, t}-H^{i, m, t}(0)$ and $G^{i, m, t}-G^{i, m, t}(0)$  against rate two Brownian motion on any fixed compact set are uniformly bounded in $i$ (and since $m$ is fixed, the same holds for the tilded versions $\widetilde{H}^{i,m,t}-\widetilde{H}^{i,m,t}(0)$ and $\widetilde{G}^{i,m,t}-\widetilde{G}^{i,m,t}(0)$). This means we can express the conditional expectation as
    \[
    \mathbb{E}\big[ \phi(Z_i)\big\vert \mathcal{F}^{i, m, t}\big] \overset d = \mathbb{E}\big[ \phi(\e_i^{-2}m^{2/3}\Gamma^{i, m, t}(1+m^{-1}\e^3_i)-\e_i^{-2}m^{2/3}\Gamma^{i, m, t}(1))\big\vert H^{i, m, t}\,, G^{i, m, t}\big]\,.
    \]
    Now, observe that for every $b> 0$, we have by the boundedness of $\phi$
    \begin{align*}
        &\mathbb{E}\big[|\mathbb{E}\big[ \phi(Z_i)\big\vert \mathcal{F}^{i, m, t}\big]-\mathbb{E}\phi(Y_1-X_1)|\big]\\
        &\le 2\sup_{x\in \R}|\phi(x)|\cdot\PP(|\widetilde{X}^{i, m, t}|\lor |\widetilde{Y}^{i, m, t}|\ge b)+\mathbb{E}\big[|\mathbb{E}\big[ \phi(Z^{b, m, t}_i)\big\vert H^{i, m, t}\,, G^{i, m, t}\big]-\mathbb{E}\phi(Y_1-X_1)|\big]\,,
    \end{align*}
    where $Z^{b, m, t}_i := \e_i^{-2}m^{2/3}\Gamma^{b,i, m, t}(1+m^{-1}\e^3_i)-\e_i^{-2}m^{2/3}\Gamma^{b,i, m, t}(1)$ and $\Gamma^{b, i, m, t}$ is the unique $\cL$ geodesic on $[0, 2+m^{-1}\e^3_i]$ with endpoints $\widetilde{X}^{b,i, m, t} := \Gamma^{b, i, m, t}(0), \widetilde{Y}^{b,i, m, t} := \Gamma^{b, i, m, t}(2+m^{-1}\e^3_i)$ given by
    \[
    (\widetilde{X}^{b, i, m, t}, \widetilde{Y}^{b, i, m, t}) = \operatorname*{argmax}_{x, y\in [-b, b]}\big(\widetilde{H}^{i, m, t}(x)+\cL(x, 0; y, 2+m^{-1}\e^3_i)+ \widetilde{G}^{i, m, t}(y)\big)\,.
    \]
    Since $m^{2/3}\widetilde{X}^{i,m, t}, m^{2/3}\widetilde{Y}^{i, m , t}$ are equal in law to $\Pi(t+\e^3_i-m), \Pi(t+2\e^3_i+m)$ respectively, we have $|\widetilde{X}^{i, m, t}|\lor |\widetilde{Y}^{i, m, t}|$ is tight in $i\ge i_*$, for some $i_*(t, m) \ge 1$ ensuring $0<t+\varepsilon_i^3-m<t+2\varepsilon_i^3+m<1$. Indeed, by Proposition 12.3 in \cite{DOV} we have 
    \begin{align*}
        \sup_{i\ge i_*}\PP(|\widetilde{X}^{i, m, t}|\lor |\widetilde{Y}^{i, m, t}|\ge b)&= \sup_{i\ge i_*}\PP(|\Pi(t+\e^3_i-m)|\lor |\Pi(t+2\e_i^3+m)|\ge bm^{2/3})\\
        &\le \PP\big(\sup_{u\in [0, 1]}|\Pi(u)|\ge bm^{2/3}\big) \to 0 \,,\qquad b \to \infty\,.
    \end{align*}
    Moreover, for any fixed $b> 0$, \cite[Corollary 6.7]{sarkarthrehalves2022} applied to the truncated profiles $(\widetilde{H}^{i, m, t})^b$, $(\widetilde{G}^{i, m, t})^b$ (recalling the notation in the remark after Lemma \ref{lemma: geod bdry cond}) with $\e^3 = m^{-1}\e^3_i$ and $t = 2$ gives
    \[
     \mathbb{E}\big[ \phi(Z^{b, m, t}_i)\big\vert H^{i, m, t}\,, G^{i, m, t}\big]\eqd \mathbb{E}\phi(Y_1-X_1)\,,\qquad i \to \infty\,,
    \]
    which by the boundedness of $\phi$ immediately yields $L^1$ convergence as $i\to \infty$. This means that
    \[
    \mathbb{E}\big[|\mathbb{E}\big[ \phi(Z^{b, m, t}_i)\big\vert H^{i, m, t}\,, G^{i, m, t}\big]-\mathbb{E}\phi(Y_1-X_1)|\big]\to 0\,,\quad  i\to \infty\,.
    \]
    We thus obtain \eqref{eq: asymp brown-bessel mixing nbhd}.

    Now, \eqref{eq: asymp brown-bessel mixing} implies the C\'{e}saro means also converge
    \[
    \lim_{n\to \infty}\mathbb{E}[G(n)\cdot H] = \E\phi(Y_1-X_1)\cdot \mathbb{E} H\,,
    \]
    for all $H$ as above.

    To conclude, take any $c\neq  \E\phi(Y_1-X_1)$, subsequence $n_k\to \infty$ as $k\to \infty$ and set $H = 1_{E_c}$ in \eqref{eq: asymp brown-bessel mixing}, with 
    \[
    E_c := \{ G(n_k) \to c \,, \quad k\to \infty\}
    \]
    to obtain $\PP(E_c) = 0$.
\end{proof}

\begin{proof}[Proof of Theorem \ref{thm: sing}] Since by Theorem \ref{thm: distinct-displacement-laws} the displacements $Y-X$ and $Y_1-X_1$ have different distributions, there exists (see Corollary \ref{cor: distinguishing-observable} for an explicit example) a bounded continuous function $\phi$ such that 
\[\E\phi(Y-X)\neq \E\phi(Y_1-X_1)\,.\]
Now this theorem follows from the above two theorems taking $c = \E\phi(Y-X)$.
\end{proof}

In the following theorem we show that indeed, the laws of $Y-X$ and $Y_1-X_1$ have different distributions. Proposition~3.9 in \cite{DDV24} supplies
a sharp local upper-tail estimate for the directed landscape.
Theorem~1.1 in \cite{DDV24}, together with the rate-function description in Section \ref{sec: prelim} (see also Section~1.1 in \cite{DDV24}) supplies a large-deviation lower bound for the probability the directed landscape takes values in a neighbourhood of an explicitly constructed deterministic metric. We explain both applications below.

\begin{theorem}\label{thm: distinct-displacement-laws}
Let $\mathcal{A}_1$ be a parabolic Airy$_2$ process and let $\mathcal{S}$ be
a standard Airy sheet. Let $W$ be a two-sided Brownian motion
with diffusion parameter $2$, and let $B$ and $R$ be a two-sided rate one Brownian motion and a two-sided
Bessel-$3$ process with diffusion parameter $1$, respectively.
Assume that these objects are mutually independent. Define
\[
(X,Y)
=
\operatorname*{arg\,max}_{(x,y)\in\mathbb R^2}
\bigl\{\mathcal{A}_1(x)-\mathcal{A}_1(0)+\mathcal{S}(x, y)+W(y)\bigr\},
\]
and
\[
(X_1,Y_1)
=
\operatorname*{arg\,max}_{(x,y)\in\mathbb R^2}
\bigl\{B(x)-R(x)+\mathcal{S}(x, y)-B(y)-R(y)\bigr\}.
\]
Call the displacements $D=Y-X$ and $D_1=Y_1-X_1$. Then
\begin{align}
\limsup_{r\to\infty}r^{-3}\log\mathbb P(|D_1|>r)
&\le -\frac{5}{12}.
\label{eq:bulk-tail-upper}\\
\liminf_{r\to\infty}r^{-3}\log\mathbb P(D>r)
&\ge -\frac{28-8\sqrt6}{25},
\label{eq:starting-tail-lower}\,.
\end{align}
In particular, $D$ and $D_1$ have different distributions.
\end{theorem}

We now start with the proof of the Brownian-Bessel environment upper tail asymptotics.

\begin{lemma}\label{lemma: brown-bessel upper bound}
    With $D_1$ as in the statement of Theorem \ref{thm: distinct-displacement-laws}, we have 
    \[
    \limsup_{r\to\infty}r^{-3}\log\mathbb P(D_1>r)\le -\frac{5}{12}\,.
    \]
\end{lemma}

\begin{proof}
Let $\mathcal{A}_1$ be a parabolic Airy$_2$ process and let $\mathcal{S}$ be a standard Airy sheet. Let $W$ be a two-sided Brownian motion with diffusion parameter $2$, and let $B$ and $R$ be,
respectively, a two-sided Brownian motion and a two-sided
Bessel-$3$ process, both with diffusion parameter $1$.

Consider the random field
\[
H_1(x,y)
=
B(x)-B(y)-R(x)-R(y)+\mathcal{S}(x, y)\,,\quad x, y \in \R\,.
\]
Since $B(0)=R(0)=0$, its almost surely unique maximizer satisfies $H_1(X_1,Y_1)\ge H_1(0,0)=\mathcal{S}(0,0)$.

We will now control the maximum of the field $H_1$ on $\R^2$ by a spatial discretisation of the plane using the unit lattice $\Z^2$. For $i,j\in\mathbb Z$, put $I_i=[i,i+1]$ and $d=|j-i|$,
and define
\[
M_{ij}
=
\sup_{x\in I_i\,,y\in I_j}
\bigl(B(x)-B(y)\bigr),
\qquad
Z_{ij}
=
\sup_{x\in I_i\,,y\in I_j}
\bigl(\mathcal{S}(x, y)+(x-y)^2\bigr),
\qquad
R_i=\inf_{x\in I_i}R(x).
\]
For $d\ge1$, non-negativity of the Bessel process $R$ gives
\begin{equation}\label{eq:bulk-box-bound}
\sup_{x\in I_i\,, y\in I_j}H_1(x, y)
\le M_{ij}+Z_{ij}-(d-1)^2-R_i.
\end{equation}

We first estimate the Brownian term $M_{ij}$. Let $O_i=\sup_{x\in I_i}|B(x)-B(i)|$. Then $M_{ij}\le B(i)-B(j)+O_i+O_j$. The increment $B(i)-B(j)$ is centered Gaussian with variance
$d$, while the unit-interval oscillations have Gaussian
exponential-moment bounds. Applying H\"older's inequality to isolate $B(i)-B(j)$ with
an exponent arbitrarily close to $1$ gives for all $\e>0$ sufficiently small,
\begin{equation}\label{eq: brownian box mgf}
\mathbb E e^{dM_{ij}}
\le
\exp\left(\left(\frac12+\e\right)d^3
+C_\e d^2+C_\e\right)\,,
\end{equation}
uniformly in $i,j$ with $d=|j-i|\ge1$ for some $\e$-dependent constant $C_\e> 0$. Indeed, for $i, j, k\ge 1$ with $G_{ij}=B(i)-B(j)$, we have the estimate $M_{ij}\le G_{ij}+O_i+O_j$. Note $G_{ij}$ is a centred Gaussian with variance
$d=|j-i|$. Consequently, we compute exactly
\[
\mathbb E e^{\lambda G_{ij}}
=\exp\left(\frac{\lambda^2d}{2}\right).
\]
By independence and stationarity of Brownian increments, we can bound the exponential moments of $O_k$, uniformly
in $k$ as follows,
\begin{equation}\label{eq: oscillation sum mgf}
\mathbb E e^{\lambda(O_i+O_j)}
\le
\left(
\mathbb E e^{2\lambda O_i}
\mathbb E e^{2\lambda O_j}
\right)^{1/2}
\le4e^{2\lambda^2}.
\end{equation}
Fix $\e>0$. H\"older's inequality with exponents 
\[
p=1+2\e,
\qquad
q=\frac{p}{p-1}
=\frac{1+2\e}{2\e}
\]
and
\eqref{eq: oscillation sum mgf} imply
\begin{align*}
\mathbb E e^{dM_{ij}}
&\le
\mathbb E\left[e^{dG_{ij}}e^{d(O_i+O_j)}\right]\le
\left(\mathbb E e^{pdG_{ij}}\right)^{1/p}
\left(\mathbb E e^{qd(O_i+O_j)}\right)^{1/q}\le
\exp\left(\frac{p}{2}d^3\right)
\left(4e^{2q^2d^2}\right)^{1/q}.
\end{align*}
Taking logarithms gives the explicit estimate
\[
\log\mathbb E e^{dM_{ij}}
\le
\left(\frac12+\e\right)d^3
+
\frac{1+2\e}{\e}\,d^2
+
\frac{2\e}{1+2\e}\log4.
\]
In particular,
\[
\log\mathbb E e^{dM_{ij}}
\le
\left(\frac12+\e\right)d^3
+C_\e d^2+C_\e,
\]
for some constant $C_\e$, uniformly in $i,j$ with $d=|j-i|\ge1$.

We now turn our attention to the terms $Z_{ij}$ in \eqref{eq:bulk-box-bound} and obtain exponential moment bounds. Consider the stationary Airy sheet (recall the spatial symmetries from \eqref{lemma: invariance dir land}),
\[
\widehat{\mathcal{S}}(x, y)=\mathcal{S}(x, y)+(x-y)^2,
\qquad
Z=\sup_{(x,y)\in[0,1]^2}\widehat{\mathcal{S}}(x, y).
\]
The above stationarity means every $Z_{ij}$ has the same law as $Z$.
With $u_0=(0,0;0,1)$,
\cite[Proposition 3.9]{DDV24} states that there is a universal
constant $c_0>0$ such that, for every $\eta\in(0,4/3)$,
there is $C_\eta<\infty$ satisfying
\begin{equation}\label{eq: dir land neighborhood tail}
\mathbb P\left(
\sup_{\|u-u_0\|_\infty\le\delta}\mathcal L(u)>h
\right)
\le
C_\eta
\exp\left(
-\bigl(\eta-c_0\delta^{1/3}\bigr)h^{3/2}
\right),
\end{equation}
for $h>0$ and $0<\delta\le1/5$. We now use the above local upper tail estimate \eqref{eq: dir land neighborhood tail} to obtain upper tail bounds for $Z$.

Now, $[0,1]^2$ using squares of width $2\delta > 0$. By stationarity of the Airy sheet, the supremum of $\widehat S$
on each such square has the same law as $Z_\delta
=
\sup_{|x|,|y|\le\delta}\widehat{\mathcal{S}}(x, y)$. Since $(x-y)^2\le4\delta^2$ there, we have $Z_\delta
\le
\sup_{|x|,|y|\le\delta}\mathcal L(x,0;y,1)
+4\delta^2$. Every point pair $(x,0;y,1)$ in this supremum belongs to the
$\delta$-neighbourhood of $u_0$. Thus a union bound and
\eqref{eq: dir land neighborhood tail} give for $h>4\delta^2$,
\[
\mathbb P(Z>h)
\le
\left\lceil\frac{1}{\delta}\right\rceil^2 C_\eta
\exp\left(
-\bigl(\eta-c_0\delta^{1/3}\bigr)
(h-4\delta^2)^{3/2}
\right).
\]
Taking first $h\to\infty$, then $\delta\downarrow0$, and finally
$\eta\uparrow4/3$ we obtain
\begin{equation}\label{eq:compact-airy-tail}
\limsup_{h\to\infty}
h^{-3/2}\log\mathbb P(Z>h)\le-\frac43.
\end{equation}
The tail estimates \eqref{eq: dir land neighborhood tail} give for every $\kappa\in(0,4/3)$, after augmenting constants,
\[
\mathbb P(Z>h)\le C_\kappa e^{-\kappa h^{3/2}},
\qquad h\ge0.
\]
We now obtain exponential moment bounds for $Z$, whose proof we relegate to the Appendix in Lemma \ref{lemma: Z exp moment bound}. In particular, we have
\[
\limsup_{\lambda\to\infty}
\lambda^{-3}\log\mathbb E e^{\lambda Z}
\le
\frac{4}{27\kappa^2}.
\]
Letting $\kappa\uparrow4/3$ yields
\begin{equation}\label{eq:airy-box-mgf}
\limsup_{\lambda\to\infty}
\lambda^{-3}\log\mathbb E e^{\lambda Z}
\le\frac1{12}.
\end{equation}

We will now treat the Bessel terms $R_i$ that appear in \eqref{eq: brownian box mgf}. We first show the summability of
\begin{equation}\label{eq:bessel-summability}
\sum_{i\in\mathbb Z}\mathbb E e^{-R_i}<\infty.
\end{equation}
Indeed, on the positive half-line we use the fact that the Bessel-$3$ process can be realised as the Euclidean norm of a standard Brownian motion. Hence, without loss of generality (since $R_i$ are mutually independent from the terms $M_{ij}$ and $Z_{ij}$), we can take
$R(t)=|\mathbf B(t)|$, where $\mathbf B$ is a standard
three-dimensional Brownian motion. For $i\ge1$, the reverse triangle inequality gives
\[
R_i
\ge
|\mathbf B(i)|
-
\sup_{0\le s\le1}|\mathbf B(i+s)-\mathbf B(i)|.
\]
The second term is independent of $\mathbf B(i)$ and has
finite exponential moments. Hence, we estimate
\[
\mathbb E e^{-R_i}
\le C\,\mathbb E e^{-|\mathbf B(i)|}
\le C i^{-3/2}\,,
\]
for some absolute constant $C>0$, which is summable. The proof for the negative half-line is identical by symmetry, proving
\eqref{eq:bessel-summability}.

Fix $\delta>0$. The GUE Tracy--Widom lower-tail estimate gives
\begin{equation}\label{eq:reference-value-tail}
\mathbb P\bigl(\mathcal{S}(0,0)<-\delta r^2\bigr)
\le C_\delta e^{-c_\delta r^6}
\end{equation}
for all sufficiently large $r$.
On the complementary event of \eqref{eq:reference-value-tail}, $S(0,0)\ge-\delta r^2$, if $|D_1|>r$, then the maximizing
pair $(X_1, Y_1)$ belongs to a box $I_i\times I_j$ with $d=|j-i|\ge r-1$,
and \eqref{eq:bulk-box-bound} implies
\[
M_{ij}+Z_{ij}
\ge (d-1)^2+R_i-\delta r^2.
\]
The variables $M_{ij}$, $Z_{ij}$, and $R_i$ are independent.
Applying the exponential Markov's inequality with parameter $d$ therefore
gives
\begin{align*}
&\mathbb P\bigl(
M_{ij}+Z_{ij}\ge(d-1)^2+R_i-\delta r^2
\bigr)\le
e^{-d(d-1)^2+\delta dr^2}
\mathbb E e^{dM_{ij}}\,
\mathbb E e^{dZ_{ij}}\,
\mathbb E e^{-dR_i}.
\end{align*}
By \eqref{eq: brownian box mgf} and \eqref{eq:airy-box-mgf},
for every $\e>0$ and all sufficiently large $d$,
this is at most
\[
\exp\left(
-\left(\frac5{12}-2\e\right)d^3
+O_\e(d^2)+\delta dr^2
\right)
\mathbb E e^{-dR_i}.
\]
Also,
$e^{-dR_i}\le e^{-R_i}$ for $d\ge 1$. Observe for $d\ge r-1$, $dr^2 \le d(d+1)^2 \le 4d^3$. Moreover, for each $i,d$, there are at most two choices of $j$.
Performing a union bound, summing over $i$ using \eqref{eq:bessel-summability}, and then
over $d\ge r-1$, yields for $0< 4\delta < \e < 5/36$,
\begin{align*}
    \mathbb P(|D_1|>r)&\le  C_\delta e^{-c_\delta r^6}+2\displaystyle \sum_{i\in \Z}\sum_{\substack{d\in \N\\ d \ge r-1}} \exp\left(
-\left(\frac5{12}-2\e\right)d^3
+O_\e(d^2)+\delta dr^2
\right)
\mathbb E e^{-dR_i}\\
&\stackrel{\eqref{eq:bessel-summability}}{\le}  C_\delta e^{-c_\delta r^6}+ 2\displaystyle \sum_{i\in \Z}\mathbb E e^{-R_i}\cdot \sum_{\substack{d\in \N\\ d \ge r-1}} \exp\left(
-\left(\frac5{12}-2\e\right)d^3
+O_\e(d^2)+\delta dr^2
\right)\\
&=  C_\delta e^{-c_\delta r^6}+2\displaystyle \sum_{i\in \Z}\mathbb E e^{-R_i}\cdot\displaystyle \sum_{\substack{d\in \N\\ d \ge r-1}} \exp\left(
-\left(\frac5{12}-3\e\right)d^3\right)\cdot\exp\left(
-\e d^3
+O_\e(d^2)+\delta dr^2
\right)\\
&\le  C_\delta e^{-c_\delta r^6}+2\displaystyle \displaystyle \sum_{i\in \Z}\mathbb E e^{-R_i}\cdot\exp\left(
-\left(\frac5{12}-3\e\right)(r-1)^3\right)\cdot\sum_{\substack{d\in \N\\ d \ge r-1}} \exp\left(
-\e d^3
+O_\e(d^2)+ 4\delta d^3
\right)\\
&\le  C_\delta e^{-c_\delta r^6}+C_{\delta, \e}\cdot\displaystyle \exp\left(
-\left(\frac{5}{12}-3\e\right)(r-1)^3\right)\,,
\end{align*}
for some constants $c_\delta, C_\delta, C_{\delta,\e}>0$. Hence, taking $r\to \infty$,
\[
\limsup_{r\to\infty}r^{-3}\log\mathbb P(|D_1|>r)
\le-\frac5{12}+3\e.
\]
The exceptional probability in \eqref{eq:reference-value-tail}
is negligible at speed $r^3$. Letting
$\e\downarrow0$ proves \eqref{eq:bulk-tail-upper}.
\end{proof}

We now move on to the proof of the Brownian-Airy environment upper tail asymptotics.

\begin{lemma}\label{lemma: brown-airy upper bound}
    With $D$ as in the statement of Theorem \ref{thm: distinct-displacement-laws}, we have 
    \[
    \liminf_{r\to\infty}r^{-3}\log\mathbb P(D>r)\ge -\frac{28-8\sqrt{6}}{25}\,.
    \]
\end{lemma}

\begin{proof}
Let $\mathcal{A}_1$ be a standard parabolic Airy$_2$ process and let $S$ be
a standard Airy sheet. Let $W$ be a two-sided Brownian motion
with diffusion parameter $2$, and let $B$ and $R$ be a two-sided Brownian motion and a two-sided
Bessel-$3$ process respectively, both with diffusion parameter $1$.

By the coupling between the Airy sheet and the parabolic Airy$_2$ process, (see Section \ref{sec: prelim}), we can realise $\mathcal{A}_1$ and $\mathcal{S}$ using one directed landscape:
\[
\mathcal{A}_1(x)=\mathcal L(0,-1;x,0),
\qquad
\mathcal{S}(x, y)=\mathcal L(x,0;y,1).
\]
These fields are independent because they concern disjoint
temporal increments, see Definition \ref{def: dir land}. The term $-\mathcal{A}_1(0)$ does not affect the
maximizer. Thus $(X,Y)$ consists of the positions at times
$0$ and $1$ of the maximizing path from $(0,-1)$ to the
terminal profile $W$.

We start with a deterministic construction of an environment on $\R^2$ ($1+1$-dimensional spacetime) and associated `metric', with maximiser $(x, y)$ has displacement $y-x=1$. By slightly perturbing this environment, the maximiser can be made unique. We will then construct a terminal boundary function $b:\R\to \R$ that will preserve the above maximiser. Finally, by dilation, we will construct environments with unique maximisers having displacement at least $a$ for any $a>1$. This will allow us to consider environments in their vicinity in the local uniform topology of the space of directed metrics $\mathcal{E}$ (see Section \ref{sec: prelim}) and use the theory of large upper tail deviations for the directed landscape established in Theorem 1.1 of \cite{DDV24} to obtain the desired asymptotic lower bounds on the upper tails of $D$. 

In other words, recalling the notation from Section \ref{sec: prelim}, we will first look for an environment $e$ with piecewise linear $e$-geodesic $f:[-1, 1]\to \R$, satisfying $f(-1) = 0, f(1)-f(0) = 1$. This will be constructed from a planted network measure $\mu$ supported on the graph of $f$.

By \cite[Proposition 6.7]{DDV24}, for any planted network measure $\mu$, the length (recall from \eqref{eq: dir metric length}) of a path $\gamma\in H^1$ is given by $\ell_{e_\mu}(\gamma) = \mu(\mathfrak{g} \gamma)-\int^t_s |\gamma'|^2$. Thus, to show that $f$ is an $e_\mu$ geodesic, it suffices to show it is a maximiser of the variational problem
\[
\sup_\gamma \mu(\mathfrak{g} \gamma)-\int^t_s |\gamma'|^2\mathrm{d} u\,,\qquad \gamma \in H^1\,,\quad  \gamma(-1) = 0\,, \gamma(1) = f(1)\,.
\]
For more examples of similar variation problems, the interested reader can see Section 2.1 in \cite{DDV24}, wherein the large deviation theory of the directed geodesic is investigated.

Since we take $f$ to be piecewise linear, the conditions above imply that on $(0, 1)$, $f'=1$. Define the measure (singular with respect to the Lebesgue measure) $\mu$ on $\R^2$ by setting
\[
\mu(E) = c\int_{-1}^1
\mathbf 1_{\{(f(t),t)\in E\}}\,dt\,,\qquad E\subset\R^2 \text{ Borel}
\]
for some $c> 0$ to be determined later. Let $e_\mu : \R^4_\uparrow \to \R$ be the metric obtained by maximizing
\[
e_\mu(x, s; y, t) := \sup_\gamma \big(\mu(\mathfrak g\gamma)-\int_s^t|\gamma'(u)|^2\,du\big)
\]
over absolutely continuous paths $\gamma$ with the prescribed
endpoints $\gamma(s) = x, \gamma(t) = y$ and finite Dirichlet energy $\int|\gamma'|^2$, where
\[
\mathfrak g\gamma=\{(\gamma(u),u):u\in[s,t]\}
\]
denotes the graph of the path $\gamma$ in the plane. This is exactly the planted-metric construction in
\cite[equation (1.4) and Definition 6.1]{DDV24}, which we briefly touch upon in Section \ref{sec: prelim} (see Figure \ref{fig: planted metric} therein). For more constructions of planted metrics and their applications, see examples 2.2 and 2.3 in \cite{DDV24}. We will also call the term (suppressing dependence on $s, t$)
\[
J_\mu(\gamma):= \mu(\mathfrak g\gamma) -\int_s^t|\gamma'(u)|^2\,du
\]
the \emph{weight} accumulated by the geodesic from $s$ to $t$.

We now show that for any $-1 \le s < t \le 1$, $f$ indeed maximises the functional $J_\mu$ over paths with finite Dirichlet energy starting at $(f(s), s)$ and ending at $(f(t), t)$ for all $-1 \le s < t\le 1$. First observe that the accumulated weight of $f$ from $s$ to $t$ is at least
\[
\mu (\mathfrak{g}f|_{[s, t]}) - \int_s^t |f'(u)|^2\,du = \int_s^t\left(c-|f'(u)|^2\right)\,du
\ge
-\frac{(f(t)-f(s))^2}{t-s}\,,
\]
provided we take 
\[
c \ge \sup_{-1\le s < t \le 1}\operatorname{Var}_{[s,t]}(f')
\]
where $\operatorname{Var}_{[s,t]}$ denotes variance with respect
to normalized Lebesgue measure on $[s,t]$: for $s<t$,
\[
\operatorname{Var}_{[s,t]}(f')
:=
\frac{1}{t-s}\int_s^t |f'(u)|^2\,du
-
\left(\frac{1}{t-s}\int_s^t f'(u)\,du\right)^2.
\]
Equivalently, this is the variance of $f'(U)$ when
$U$ is uniform on $[s,t]$. By the choice of piecewise linear `planted' path $f$ with unit displacement on $[0, 1]$, on any interval $[s,t]\subset[-1,1]$, the derivative $f'$ of the linear parts
takes values in $\{\alpha,1\}$ for some $\alpha\in \R$. We thus have with
\[
p=\frac{\operatorname{Leb}\{u\in[s,t]:f'(u)=1\}}{t-s},
\]
the uniform in $s$ and $t$ bounds
\[
\operatorname{Var}_{[s,t]}(f')
=
p(1-p)\left(1-\alpha\right)^2
\le
\frac14\left(1-\alpha\right)^2
\]
and so we take $c = c_\alpha := (1-\alpha)^2/4$. Finally, any excursion away from $f$ receives no contribution from $\mu$ during that
excursion, and its weight is at most the negative Dirichlet
energy of the straight segment joining its endpoints (again by Cauchy-Schwarz).
Replacing each excursion by the corresponding subpath of
$f$ thus cannot decrease $J_\mu$.

Now, the aforementioned constraints we imposed on $f$ mean we can express all admissible $f$ as
\[
f(-1)=0,\qquad
f(0)=\alpha,\qquad
f(1)= 1+ \alpha,\qquad
f(1)-f(0)=1\,,
\]
\begin{equation}\label{eq: planted path}
f(t)=
\begin{cases}
\alpha(t+1),&-1\le t\le0,\\
\alpha + t,&0\le t\le1
\end{cases}
\end{equation}
and let $\mu$ be the planted network measure supported on the graph of $f$ with $c = c_\alpha:=(1-\alpha)^2/4$. The accumulated weight of $f$ is easily computed as
\begin{equation}\label{eq:cumulative weight}
w(t)=
\begin{cases}
 \big(\frac14 (1-\alpha)^2-\alpha^2 \big)(t+1),&-1\le t\le0,\\
 \frac14 (1-\alpha)^2(t+1)-\alpha^2-t,&0\le t\le1.
\end{cases}
\end{equation}

We will now construct a terminal profile $b: \R\to\R$ (crucially with finite Dirichlet energy on $\R$) so that $f$ is still an $e_\mu$-geodesic from $(0, -1)$ to $(b, 1)$. As before, to every path $\gamma$ with finite Dirichlet energy and $\gamma(-1) = 0$, we associate the \emph{weight}
\[
J_\mu(\gamma):= \mu(\mathfrak g\gamma)-\int_{-1}^1|\gamma'(u)|^2\,du+b(\gamma(1))
\]
(here we absorb in a slight abuse of notation the terminal value into the original definition).

Fix $s\in [-1, 1]$. Observe that any absolutely continuous path $\gamma: [-1, 1]\to \R$ with finite Dirichlet energy, boundary conditions $\gamma(-1) = f(-1) = 0$ and terminal value $\gamma(1) = y\neq f(1)$ has by continuity some final time $s\in [-1, 1)$ such that $\gamma(s) = f(s)$. Now, since $\gamma$ avoids $f$ on $(s, 1]$, we have $\mu(\mathfrak{g}\gamma|_{(s, 1]}) = 0$ and so the accumulated weight of $\gamma$ on $(s, 1]$ is bounded by 
\[
\mu(\mathfrak g \gamma|_{(s, 1]})-\int_s^1|\gamma'(u)|^2\,du = -\int_s^1|\gamma'(u)|^2\,du \le -\frac{1}{1-s}\left(\int^1_s \gamma'(u)\, \mathrm{d}u\right)^2 = -\frac{(y-f(s))^2}{1-s}
\]
where the last inequality follows by Cauchy-Schwarz, giving the desired estimate. Thus, by additivity of $J_\mu$ with respect to concatenation of paths, the total accumulated weight of $\gamma$ on $[-1, 1]$ is at most
\[
 w(s) -\frac{(y-f(s))^2}{1-s} + b(y)\,,
\]
since the weight accumulated by $\gamma$ up to time $s$ is at most $w(s)$,
because $f|_{[-1,s]}$ is an $e_\mu$-geodesic from $(0, -1)$ to $(f(s), s)$. For $f$, the accumulated weight $J_\mu(f)$ is $w(1) +b(1+\alpha)$. By \eqref{eq:cumulative weight}, we have $w(1) = -(1+\alpha)^2/2$ and $f(1) = \alpha + 1$.

To show $f$ maximises the functional $J_\mu$ over all $\gamma:[-1, 1] \to \R$ with finite Dirichlet energy and boundary condition $\gamma(-1) = 0$ (the case $\gamma(1) = f(1)$ being taken care of the fact that $f$ is an $e_\mu$ geodesic with $w(1) = -(1+\alpha)^2/2$), the previous observation means it will suffice to find a terminal profile $b$ such that $b(1+\alpha) = (1+\alpha)^2/2$ and
\begin{equation}\label{eq:terminal profile bound}
w(s)-\frac{(y-f(s))^2}{1-s}+b(y)\le0
\qquad
\text{for all }-1\le s<1,\ y\in\mathbb R\,.
\end{equation}

Notice that for $\alpha \in [1/3, 1)$ we have for all $s\in [-1, 1]$ the weight $w(s)$ of the geodesic $f$ from $(0, -1)$ to $(f(s), s)$, \eqref{eq:cumulative weight} is less than or equal to zero. Hence $b(y) = 0$ for $y\le 0$ satisfies \eqref{eq:terminal profile bound}. For $y > 0$, unravelling \eqref{eq:terminal profile bound}, we arrive at two bounds. For $s\le0$, with $h=1-s\in[1,2]$ it becomes
 \begin{align}\label{eq: upper bound 1}
b(y)&\le \frac{3\alpha^2+2\alpha-1}{4}(2-h)
+\frac{(y-\alpha(2-h))^2}{h}\\
&=
\frac{(y-2\alpha)^2}{h}
+\frac{(1-\alpha)^2}{4}h
+2\alpha y-2\alpha^2-\frac{(1-\alpha)^2}{2}:= A_\alpha(y, h)\,.
\end{align}
For $s\ge0$, with $h'=1-s\in(0,1]$, \eqref{eq:terminal profile bound} becomes
\begin{align}\label{eq: upper bound 2}
    b(y)\le \frac{(1+\alpha)^2}{2}
+2(y-1-\alpha)
+\frac{(y-1-\alpha)^2}{h'}
+\frac{(1-\alpha)^2}{4}h':=B_\alpha(y, h')\,.
\end{align}
Optimising the above right hand sides over arbitrary $h, h' > 0$
gives (keeping in mind the restriction $\alpha < 1$)
\[
h(y)=\frac{2|y-2\alpha|}{1-\alpha}\,,
\qquad
h'(y)=\frac{2|y-1-\alpha|}{1-\alpha}.
\]
We will not directly substitute the unconstrained minimisers in \eqref{eq: upper bound 1} and \eqref{eq: upper bound 2}. Instead, we will do some casework on $\alpha\in [1/3, 1]$ to find the minimisers $h, h'$ in $[1, 2]$ and $(0, 1]$ respectively. 

Fix $1/3\le\alpha<1$. For $0\le y\le1+\alpha$,
the minimizer in the right hand side of \eqref{eq: upper bound 1} over $h\in[1,2]$ is
\[
h(y)=
\begin{cases}
2,
&0\le y\le3\alpha-1,\\[4pt]
\dfrac{2(2\alpha-y)}{1-\alpha},
&3\alpha-1\le y\le\dfrac{5\alpha-1}{2},\\[8pt]
1,
&\dfrac{5\alpha-1}{2}\le y\le\dfrac{3\alpha+1}{2},\\[8pt]
\dfrac{2(y-2\alpha)}{1-\alpha},
&\dfrac{3\alpha+1}{2}\le y\le1+\alpha.
\end{cases}
\]
Substituting these minimizing values into $A_\alpha(y,h)$ gives
\begin{align}\label{eq: A minimal}
    \min_{1\le h\le2}A_\alpha(y,h)=
\begin{cases}
\dfrac{y^2}{2},
&0\le y\le3\alpha-1,\\[6pt]
(3\alpha-1)y-\dfrac{(3\alpha-1)^2}{2},
&3\alpha-1\le y\le\dfrac{5\alpha-1}{2},\\[8pt]
\dfrac{3\alpha^2+2\alpha-1}{4}+(y-\alpha)^2,
&\dfrac{5\alpha-1}{2}\le y\le\dfrac{3\alpha+1}{2},\\[8pt]
(1+\alpha)y-\dfrac{(1+\alpha)^2}{2},
&\dfrac{3\alpha+1}{2}\le y\le1+\alpha.
\end{cases}
\end{align}
For $\alpha=1$ (which incidentally corresponds to $e_\mu$ being the Dirichlet metric from Section \ref{sec: prelim}), direct minimization gives
\[
\min_{1\le h\le2}A_1(y,h)=\frac{y^2}{2},
\qquad 0\le y\le2,
\]
consistent with \eqref{eq: A minimal} after omitting
the zero-length intervals.

Now, we check that the (optimised in $h'$) values $B_\alpha$ in the right hand side of \eqref{eq: upper bound 2} are at least those obtained for $A_\alpha$. For $\alpha \in [1/3, 1)$ observe that if $y\le(3\alpha+1)/2$, then $h'\ge1$, so the
minimum over $(0,1]$ in \eqref{eq: upper bound 2} occurs at $h'(y)=1$. If $(3\alpha+1)/2<y<1+\alpha$, then $h_*\in(0,1)$,
and
\[
\frac{(y-1-\alpha)^2}{h'}
+\frac{(1-\alpha)^2}{4}h'
=
(1-\alpha)(1+\alpha-y)\,.
\]
At $y=1+\alpha$, the infimum is approached as
$h\downarrow0$. When $\alpha=1$, the upper bound in \eqref{eq: upper bound 2} is directly obtained as $1+(y-1)^2$ for $0\le y \le 2$. Therefore, we have for $\alpha \in [1/3, 1)$,
\[
\inf_{0<h\le1}B_\alpha(y,h)
=
\begin{cases}
\dfrac{3\alpha^2+2\alpha-1}{4}+(y-\alpha)^2,
&0\le y\le\dfrac{3\alpha+1}{2},\\[8pt]
(1+\alpha)y-\dfrac{(1+\alpha)^2}{2},
&\dfrac{3\alpha+1}{2}\le y\le1+\alpha\,,
\end{cases}
\]
which by inspection is seen to be at least the minimal values attained by $A_\alpha$ in \eqref{eq: A minimal}. For $\alpha=1$, direct minimization gives
\[
\inf_{0<h\le1}B_1(y,h)=1+(y-1)^2,
\qquad 0\le y\le2\,,
\]
which is seen to be at least $\min_{1\le h\le2}A_1(y,h)$ for $0\le y \le 2$. For $y\ge 1+\alpha$, the bounds above are at least $(1+\alpha)^2/2$. Thus, piecing all of the above together, the following continuous, non-decreasing terminal profile
\begin{equation}\label{eq:terminal profile}
b(y)=
\begin{cases}
0,
&y\le0,\\[2pt]
\dfrac{y^2}{2},
&0\le y\le3\alpha-1,\\[6pt]
(3\alpha-1)y-\dfrac{(3\alpha-1)^2}{2},
&3\alpha-1\le y\le\dfrac{5\alpha-1}{2},\\[8pt]
\dfrac{3\alpha^2+2\alpha-1}{4}+(y-\alpha)^2,
&\dfrac{5\alpha-1}{2}\le y\le\dfrac{3\alpha+1}{2},\\[8pt]
(1+\alpha)y-\dfrac{(1+\alpha)^2}{2},
&\dfrac{3\alpha+1}{2}\le y\le1+\alpha,\\[8pt]
\dfrac{(1+\alpha)^2}{2},
&y\ge1+\alpha.
\end{cases}
\end{equation}

We now make a slight perturbation to the environment $e_\mu$, while still preserving the geodesic property of $f$, enforcing that it is the \emph{unique} path attaining maximal weight with the prescribed boundary conditions.

Fix $\eta > 0$. Change the definition of $\mu$ by a constant factor so that for any Borel measurable subset $E$ of $\R^2$
\[
\mu_{c_\alpha+\eta}(E) : =
\left(\frac{(1-\alpha)^2}{4}+\eta\right)\int_{-1}^1
\mathbf 1_{\{(f(t),t)\in E\}}\,dt\,,
\]
and consider the environment $e_{c_\alpha+\eta} := e_{\mu_{c_\alpha+\eta}}$ with associated weight function $J_{c_\alpha+\eta}$. For finite Dirichlet energy path $\gamma$ with $\gamma(-1) = 0$, the weight now becomes
\[
J_{c_\alpha+\eta}(\gamma)
=
J_{c_\alpha}(\gamma)
+
\eta\,\operatorname{Leb}(C_\gamma)\,,
\]
where $C_\gamma = \{s\in [-1, 1] \gamma(s) = f(s)\}$. Since $C_f=[-1,1]$, we have $J_{c_\alpha+\eta}(f)=2\eta$. If $\gamma\ne f$, continuity implies that $\gamma$ and $f$ differ on a subinterval of positive length. Hence $\operatorname{Leb}(C_\gamma)<2$, and therefore (using the already established bound $J_{c_\alpha}(\gamma) \le 0$)
\[
J_{c_\alpha+\eta}(\gamma) = J_{c_\alpha}(\gamma)+\eta \operatorname{Leb}(C_\gamma)
\le
\eta\,\operatorname{Leb}(C_\gamma)
<
2\eta
=
J_{c_\alpha+\eta}(f).
\]
Thus for each fixed $\eta>0$, the preceding argument establishes
that $f$ is the unique $J_{c_\alpha + \eta}$-maximizing path for the perturbed
metric $e_{c_\alpha+\eta}$. Since the rate function of $e_{c_\alpha+\eta}$ is finite (see Section \ref{sec: prelim}), Definition 4.1 and Lemma 6.8 in \cite{DDV24}, the \emph{composition law} for $e_{c_\alpha+\eta}$ is satisfied giving
\begin{align*}
    e_{c_\alpha+\eta}(0, -1; y, 1) =  \max_{x \in \R}e_{c_\alpha+\eta}(0, -1; x, 0)+e_{c_\alpha+\eta}(x, 0; y, 1)\,,\qquad y \in \R\,.
\end{align*}
We thus have
\begin{align*}
    &\max_{x, y \in \R}e_{c_\alpha+\eta}(0, -1; x, 0)+e_{c_\alpha+\eta}(x, 0; y, 1)+b(y)\\
    = &\max_{\gamma, \tilde{\gamma}}-\int_{-1}^0|\gamma'(u)|^2\,du-\int_{0}^1|\tilde{\gamma}'(u)|^2\,du +\mu(\mathfrak g\gamma)+\mu(\mathfrak g\tilde{\gamma})+b(\tilde{\gamma}(1))\,,
\end{align*}
where the supremum is over paths $\gamma, \tilde{\gamma}$ in $H^1$ with $\gamma(-1) = 0, \gamma(0) = \tilde{\gamma}(0)$. By additivity of $\mu$ and the fact that concatenation of such paths $\gamma, \tilde{\gamma}$ is again a path $\pi$ in $H^1$ with $\pi(-1) = 0$, we obtain by the previous considerations that
\begin{align*}
    &\max_{x, y \in \R}e_{c_\alpha+\eta}(0, -1; x, 0)+e_{c_\alpha+\eta}(x, 0; y, 1)+b(y) = \max_{\gamma}J_{c_\alpha+\eta}(\gamma)\le J_{c_\alpha+\eta}(f)\,.
\end{align*}
We thus have 
\[
(f(0), f(1)) = \argmax_{x, y\in \R} e_{c_\alpha+\eta}(0, -1; x, 0)+e_{c_\alpha+\eta}(x, 0; y, 1)+b(y)
\]
noting that $f(1) -f(0) = 1+\alpha - \alpha = 1$.

We will now dilate the above construction by a scale parameter $a> 1$ so enforce a displacement of the maximiser that is \emph{strictly} greater one. This will allow us to find a neighbourhood of environments with such a displacement above one, and then apply the large deviation lower bound from \cite[Theorem 1.1]{DDV24}.

Fix $a>1$. Dilate the construction above spatially by
$a$, setting
\[
f_a(t)=af(t),
\qquad
b_a(y)=a^2b(y/a),
\]
and define
\[
\mu_{a,\eta}(E)
=
a^2\left(\frac{(1-\alpha)^2}{4}+\eta\right)
\int_{-1}^1
\mathbf 1_{\{(f_a(t),t)\in E\}}\,dt.
\]
This measure is supported on the graph of one absolutely
continuous path and has finite Kruzhkov entropy (see the introduction of \cite{DDV24} and Section \ref{sec: prelim} herein), so it is a planted
network measure in the sense of \cite[Section 1.1]{DDV24}.
Let $e_{a,\eta}=e_{\mu_{a,\eta}}$, suppressing the dependence on $\alpha$ for ease of notation. 

\underline{Claim}: With metric $e_{a, \eta}$ as above, $f_a$ is so that
\[
(a\alpha, a(1+\alpha)) = (f_a(0), f_a(1)) = \argmax_{x, y\in \R} e_{a, \eta}(0, -1; x, 0)+e_{a, \eta}(x, 0; y, 1)+b_a(y)
\]
with the pair $(a\alpha, a(1+\alpha))$ being \emph{unique}.

To see this, let $\gamma:[-1,1]\to\mathbb R$ be any
absolutely continuous path with finite Dirichlet energy and $\gamma(-1)=0$.
Set $\widetilde\gamma(t) = \gamma(t)/a$. Its weight in the dilated, perturbed environment is
\[
J_{a,\eta}(\gamma)
=
-\int_{-1}^1|\gamma'(t)|^2\,dt
+
a^2\left(\frac{(1-\alpha)^2}{4}+\eta\right)
\int_{-1}^1
\mathbf 1_{\{\gamma(t)=f_a(t)\}}\,dt
+
b_a(\gamma(1)).
\]
Since $f_a=af$ and $b_a(y)=a^2b(y/a)$, we have
\[
|\gamma'|^2=a^2|\widetilde\gamma'|^2,
\qquad
\{\gamma(t)=f_a(t)\}
=
\{\widetilde\gamma(t)=f(t)\},
\qquad
b_a(\gamma(1))=a^2b(\widetilde\gamma(1)).
\]
Therefore $J_{a,\eta}(\gamma)
=
a^2J_{c_\alpha+\eta}(\widetilde\gamma)$,
where $J_{c_\alpha+\eta}$ is the weight in the undilated, perturbed environment.

The map $\gamma\mapsto\widetilde\gamma=\gamma/a$ is clearly a
bijection between finite Dirichlet energy paths on $[-1, 1]$ vanishing at $-1$. Since $f$ is the unique $e_{c_\alpha+\eta}$-geodesic with endpoints $(0, -1)$ and $(f(1), 1)$,  $f_a=af$ is the unique path maximising $J_{a, \eta}$ and so by Proposition 6.9 in \cite{DDV24}, $f_a$ is the unique $e_{a, \eta}$-geodesic from $(0, -1)$ to $(af(1), 1)$. Its displacement on $[0,1]$ is
\[
f_a(1)-f_a(0)
=
a\left(1+\alpha - \alpha\right)
=a>1.
\]
Arguing exactly as we did for the metric $e_{c_\alpha+\eta}$, we conclude the proof of the claim.

We are now in a position to estimate from below the upper tail events $\{D > r\}$ with $D = Y-X$ the displacement of the maximisers $X, Y$. First, observe that we can express
\[
\{D> r\} = \left\{\max_{x, y \in \R}\mathcal{L}(0, -1; x, 0) + \mathcal{L}(x, 0; y, 1)+ W(y) > \max_{\substack{x, y \in \R\\y-x \le r}}\mathcal{L}(0, -1; x, 0) + \mathcal{L}(x, 0; y, 1)+ W(y)\right\}\,.
\]
By invariance under diffusive scaling, we have $W^{(r)}(y):= r^{-2}W(ry)$ is equal in law to $r^{-\frac{3}{2}}W(y)$ on $\mathcal{C}$. Now, with $\mathcal L^{(r)}(x,s;y,t) = r^{-2}\mathcal L(rx,s;ry,t)$, we can succinctly re-express
\begin{align*}
    \{D>r\} =\left\{
\begin{aligned}
&\max_{x,y\in \R}
\mathcal L^{(r)}(0,-1;x,0)
+\mathcal L^{(r)}(x,0;y,1)
+W^{(r)}(y)\\
&\qquad>
\max_{\substack{x,y\in \R\\y-x\le1}}
\mathcal L^{(r)}(0,-1;x,0)
+\mathcal L^{(r)}(x,0;y,1)
+W^{(r)}(y)
\end{aligned}
\right\}\,.
\end{align*}
Moreover, for any $K>0$ we can section on the event the rescaled maximisers $(X/r, Y/r)$ are located in the box $Q_K = [-K, K ]^2$. We can thus express on the event $\{|X|\lor |Y|\le rK\}$, 
\begin{align}
    \{D>r\} =\left\{
\begin{aligned}\label{eq: maximiser loc}
&\max_{(x,y)\in Q_K}
\mathcal L^{(r)}(0,-1;x,0)
+\mathcal L^{(r)}(x,0;y,1)
+W^{(r)}(y)\\
&\qquad>
\max_{\substack{(x,y)\in Q_K\\y-x\le1}}
\mathcal L^{(r)}(0,-1;x,0)
+\mathcal L^{(r)}(x,0;y,1)
+W^{(r)}(y)
\end{aligned}
\right\}\,.
\end{align}

Now, by \cite[Theorem 1.1]{DDV24} and Schilder's theorem respectively, the families $\mathcal L^{(r)}$ and $r^{-\frac{3}{2}}W(y)$ satisfy large deviation principles with speed $r^{3}$ (with good rate functions which we will soon compute). Having localised the maximiser by sectioning with $\{|X|\lor |Y|\le rK\}$, we are in good shape to apply large deviation theory on neighbourhoods of the dilated and perturbed environment $e_{a, \eta}$ with boundary term $b_a$, in the respective uniform topologies of directed metrics and continuous functions.

Choose $K>a(1+\alpha)$ and define the deterministic environment
\[
F_{a,\eta}(x,y)=e_{a,\eta}(0,-1;x,0)+e_{a,\eta}(x,0;y,1)+b_a(y)\,,\qquad x, y\in \R\,.
\]
By the preceding discussion, we have that its unique maximiser is $p_*=\left(a\alpha, a(1+\alpha)\right)$. Since the displacement at $p_*$ is $a>1$, the continuity of the environment $F_{a, \eta}$ and standard compactness argument imply
\[
g_K
:=
F_{a,\eta}(p_*)
-
\sup_{\substack{(x,y)\in Q_K\\y-x\le1}}
F_{a,\eta}(x,y)
>0.
\]
Choose $0<\e<g_K/6$. Let $U_\e$ be the set of metrics $e$ satisfying
\[
\sup_{|x|\le K}
\left|
e(0,-1;x,0)-e_{a,\eta}(0,-1;x,0)
\right|<\e\qquad \text{and} \qquad \sup_{(x,y)\in Q_K}
\left|
e(x,0;y,1)-e_{a,\eta}(x,0;y,1)
\right|<\e\,.
\]
By definition, $U_\e$ is an open neighbourhood of $e_{a,\eta}$ in $\mathcal{E}$ for the topology of uniform convergence on bounded sets used in \cite{DDV24}.
Moreover, denote the uniform neighbourhood of $b_a$ of diameter $\e$ by
\[
V_\e
=
\left\{
h\in \mathcal{C}[-K,K]:
\sup_{|y|\le K}|h(y)-b_a(y)|<\e
\right\}.
\]
For $e\in U_\e$ and $h\in V_\e$, the environment
\[
(x,y)\longmapsto
e(0,-1;x,0)+e(x,0;y,1)+h(y)
\]
differs from $F_{a,\eta}$ by less than $3\e$ on $Q_K$.
Since $6\e<g_K$, every maximizer over $Q_K$ has
displacement greater than $1$.

The lower bound in \cite[Theorem 1.1]{DDV24} thus gives
\[
\liminf_{r\to\infty}r^{-3}
\log\mathbb P(\mathcal L^{(r)}\in U_\e)
\ge
-\inf_{e\in U_\e}I(e)
\ge-I(e_{a,\eta})\,,
\]
where $I(\cdot)$ is the rate-function description in \cite[Theorem 1.1, Section 1.1]{DDV24} (see Section \ref{sec: prelim},
and is given by
\[
I(e_{a, \eta})=\frac43\int\rho_{a, \eta}(t)^{3/2}\,dt\,,
\]
where $\rho_{a, \eta}$ is the Lebesgue density of the time marginal of $\mu_{a, \eta}$ (in fact one can view $\mu_{a, \eta}$ as the pushforward of the measure on $\R$ with Lebesgue density $\rho_{a, \eta}$ under the graph of $f_a$). The landscape rate for the perturbed and dilated metric $e_{a, \eta}$ is thus computed as
\[
I(e_{a,\eta})
=
\frac43\int_{-1}^1
\left[
a^2\left(\frac{(1-\alpha)^2}{4}+\eta\right)
\right]^{3/2}\,dt
=
a^3\frac83\left(\frac{(1-\alpha)^2}{4}+\eta\right)^{3/2}\,.
\]

Schilder's theorem applied to the independent $W|_{[-K, 0]}$ and $W|_{[0, K]}$ separately, splitting the neighbourhood $V_\e$ into $\mathcal{C}[-K, 0]$ and $\mathcal{C}[0, K]$ neighbourhoods gives (taking into account that $W$ is a rate two Brownian motion and noting that $b_a(0) = 0$)
\[
\liminf_{r\to\infty}r^{-3}
\log\mathbb P(W^{(r)}|_{[-K,K]}\in V_\e)
\ge
-\frac14\int_{-K}^K|b_a'(y)|^2\,dy\,.
\]
A quick computation gives the rate function with $b_a$ the dilated version of \eqref{eq:terminal profile} as (using that $[-K, K]$ contains the support of $b_a'$)
\begin{align*}
\frac14\int_{-K}^K|b_a'(y)|^2\,dy&=\frac{a^3}{4}\int_{\mathbb R}|b'(y)|^2\,dy\\
&=\frac{a^3}{4}\bigg[
\int_0^{3\alpha-1}y^2\,dy+
\int_{3\alpha-1}^{(5\alpha-1)/2}
(3\alpha-1)^2\,dy\\
&+
\int_{(5\alpha-1)/2}^{(3\alpha+1)/2}
4(y-\alpha)^2\,dy+
\int_{(3\alpha+1)/2}^{1+\alpha}
(1+\alpha)^2\,dy
\bigg]\\
&=
a^3\left(\frac{(3\alpha-1)^3+(1+\alpha)^3}{24}+
\frac{1-\alpha}{8}
\big((3\alpha-1)^2+(1+\alpha)^2\big)\right)\\
&=
a^3\left(\frac14-\frac{\alpha}{4}
+\frac{3\alpha^2}{4}-\frac{\alpha^3}{12}\right)\,.
\end{align*}

By independence, the event $E_r
:=
\{\mathcal L^{(r)}\in U_\e\}
\cap
\{W^{(r)}|_{[-K,K]}\in V_\e\}$ satisfies the asymptotic lower bound
\begin{equation}\label{eq:local-event-lower}
\liminf_{r\to\infty}r^{-3}\log\mathbb P(E_r)\ge-J_{a,\eta}\,,
\end{equation}
with
\begin{equation}\label{eq:total-perturbed-rate}
J_{a,\eta}
=
a^3\left[
\frac83\left(\frac{(1-\alpha)^2}{4}+\eta\right)^{3/2}
+\frac14-\frac{\alpha}{4}
+\frac{3\alpha^2}{4}-\frac{\alpha^3}{12}
\right]\,.
\end{equation}

By \eqref{eq: maximiser loc} and the above, we have the inclusion $E_r \cap \{|X|\lor |Y|\le rK\}\subseteq \{D> r\}$. We thus estimate
\begin{align}\label{eq: D lower bound}
    \mathbb \PP(D>r) \ge \mathbb \PP(E_r)-\PP(|X|\lor |Y|\ge rK)\,.
\end{align}
The event $\PP(|X|\lor |Y|\ge rK)$ is \emph{negligible} in comparison to $\PP(E_r)$ for $K $sufficiently large. Indeed, arguing as in the proof of Lemma \ref{lemma: geod endpts tight}, replacing $f_\e$ therein with the Airy$_2$ process and estimating the quadratic term 
\[
x^2 + (x-y)^2 \ge c(x^2+y^2)\,,\qquad x, y \in \R\,,
\]
for some universal constant $c> 0$, we obtain the estimates (crucially note the \emph{cubic} decay)
\begin{equation}\label{eq:starting-localization}
\mathbb P(|X|\vee|Y|>u)\le Ce^{-cu^3},
\qquad u\ge1,
\end{equation}
for some universal constant $C>0$. By \eqref{eq:starting-localization}, we now have
\[
\mathbb \PP(D>r)
\ge
\mathbb \PP(E_r)-Ce^{-cK^3 r^3}\,.
\]
Choose $K$ sufficiently large that $cK^3>J_{a,\eta}$. Taking $r\to \infty$ and using \eqref{eq:local-event-lower}, we obtain for all $\alpha \in [1/3, 1]$,
\[
\liminf_{r\to\infty}r^{-3}\log\mathbb P(D>r)
\ge
-a^3\left[
\frac83\left(\frac{(1-\alpha)^2}{4}+\eta\right)^{3/2}
+\frac14-\frac{\alpha}{4}
+\frac{3\alpha^2}{4}-\frac{\alpha^3}{12}
\right].
\]
Letting $\eta\downarrow0$ and then $a\downarrow1$ gives
\[
\liminf_{r\to\infty}r^{-3}\log\mathbb P(D>r)
\ge
-\left(
\frac7{12}-\frac{5\alpha}{4}
+\frac{7\alpha^2}{4}-\frac{5\alpha^3}{12}
\right)\,.
\]
Define
\[
F(\alpha)
=
\frac7{12}-\frac{5\alpha}{4}
+\frac{7\alpha^2}{4}-\frac{5\alpha^3}{12},
\qquad \alpha\in[1/3,1].
\]
Differentiating gives
\[
F'(\alpha)=\frac{-5+14\alpha-5\alpha^2}{4},
\qquad
F''(\alpha)=\frac{7-5\alpha}{2}>0
\quad\text{on }[1/3,1].
\]
Thus $F$ is strictly convex on this interval. Solving $F'(\alpha)=0$ yields
\[
5\alpha^2-14\alpha+5=0,
\qquad
\alpha=\frac{7\pm2\sqrt6}{5}.
\]
Only $\alpha^* := (7-2\sqrt6)/5$ belongs to $[1/3,1]$, so it minimises $F$ on $[1/3, 1]$ by convexity. Substituting back into $F$ gives
\[
\min_{\alpha\in[1/3,1]}F(\alpha)
=
F(\alpha_*)
=
\frac16 (1+\alpha_*)^2
=
\frac{28-8\sqrt6}{25}
\]
We thus obtain
\[
\liminf_{r\to\infty}r^{-3}\log\mathbb P(D>r)
\ge - F(\alpha^*) = -\frac{28-8\sqrt6}{25}.
\]
proves
\eqref{eq:starting-tail-lower}.
\end{proof}

\begin{corollary}\label{cor: distinguishing-observable}
Under the assumptions of
Theorem~\ref{thm: distinct-displacement-laws}, define
\[
\phi_r(z)=\min\{1,(|z|-r)_+\}.
\]
For all sufficiently large $r$,
\[
\mathbb E\phi_r(D)>\mathbb E\phi_r(D_1).
\]
\end{corollary}

\begin{proof}
The function $\phi_r$ is bounded and continuous, and
\[
\mathbb E\phi_r(D)\ge\mathbb P(D>r+1),
\qquad
\mathbb E\phi_r(D_1)\le\mathbb P(|D_1|>r).
\]
The conclusion follows from
\eqref{eq:starting-tail-lower} and \eqref{eq:bulk-tail-upper},
since
\[
\frac5{12}-\frac{28-8\sqrt6}{25}=\frac{125+96\sqrt{6}-336}{300}>0.
\]
\end{proof}

\section{Appendix}

The following lemma, establishes asymptotic exponential moment bounds for a random variables with prescribed stretched exponential upper tail estimates, ussed in the proof of Lemma \ref{lemma: brown-bessel upper bound}.

\begin{lemma}\label{lemma: Z exp moment bound}
    Fix $\kappa \in (0, \infty)$. Suppose $Z$ is a real-valued random variable with upper tail estimates
    \[
        \mathbb P(Z>h)\le C_\kappa e^{-\kappa h^{3/2}},
        \qquad h\ge0.
    \]
    Then, we have the asymptotic exponential moment bounds,
    \[
    \limsup_{\lambda\to\infty}
    \lambda^{-3}\log\mathbb E e^{\lambda Z}
    \le
    \frac{4}{27\kappa^2}.
    \]
\end{lemma}

\begin{proof}
Fix $\lambda > 0$. Tonelli's theorem gives the bounds
    \[
\mathbb E e^{\lambda Z} \le 1+C_\kappa\lambda \int_0^\infty e^{\lambda h-\kappa h^{3/2}}\,\mathrm{d}h.
\]
The change of variables $h=\lambda^2v$ and the above estimates give
\begin{equation}\label{eq: integral bound exp moment}
\mathbb E e^{\lambda Z}
\le
1+C_\kappa\lambda^3
\int_0^\infty \mathrm{e}^{\lambda^3(v-\kappa v^{3/2})}\mathrm{d}v.
\end{equation}
We now estimate the integral in \eqref{eq: integral bound exp moment}. We will henceforth enlarge or reduce any constants as needed. Consider the function $f(v)=v-\kappa v^{3/2}$. For $v>0$, its derivative $f'(v)=1-\frac{3\kappa}{2}\sqrt v$. Thus $f$ increases up to $v_*=\frac{4}{9\kappa^2}
$ and decreases thereafter. Its maximum is
\[
m:=f(v_*)
=
\frac{4}{9\kappa^2}
-\kappa\frac{8}{27\kappa^3}
=
\frac{4}{27\kappa^2}>0\,.
\]
Set $V= 1 \lor (4/\kappa^2)$. On $[0,V]$, we have $f(v)\le m$, so we estimate
\[
\int_0^V e^{\lambda^3f(v)}\,\mathrm{d}v
\le V e^{m\lambda^3}.
\]
For $v\in [V, \infty)$ the inequality
$\sqrt v\ge2/\kappa$ implies
\[
f(v)
=v-\kappa v^{3/2}
\le-\frac{\kappa}{2}v^{3/2}
\le-\frac{\kappa}{2}v.
\]
Therefore
\[
\int_V^\infty e^{\lambda^3f(v)}\,\mathrm{d}v
\le
\int_V^\infty e^{-(\kappa/2)\lambda^3v}\,\mathrm{d}v
=
\frac{2}{\kappa\lambda^3}
e^{-(\kappa/2)\lambda^3V}.
\]
Combining the two bounds, we obtain
\[
\mathbb E e^{\lambda Z}
\le
1+C_\kappa V\lambda^3e^{m\lambda^3}
+\frac{2C_\kappa}{\kappa}
e^{-(\kappa/2)\lambda^3V}.
\]
Since $m>0$, for $\lambda\ge1$ the right-hand side is at most $C_\kappa(1+\lambda^3)e^{m\lambda^3}$. 
Consequently,
\[
\lambda^{-3}\log\mathbb E e^{\lambda Z}
\le
m+
\frac{\log C_\kappa+\log(1+\lambda^3)}{\lambda^3}.
\]
The second term tends to zero as $\lambda \to \infty$, finally establishing
\[
\limsup_{\lambda\to\infty}
\lambda^{-3}\log\mathbb E e^{\lambda Z}
\le
\frac{4}{27\kappa^2}.
\]
\end{proof}

\bibliographystyle{alpha}
\bibliography{bibliography.bib}

\end{document}